\documentclass[a4paper, 10pt, reqno]{amsart}

\usepackage[utf8]{inputenc}
\usepackage[T1]{fontenc}
\usepackage{microtype}
\usepackage{lmodern}
\usepackage{amsmath}
\usepackage{amsthm}
\usepackage{amssymb}
\usepackage{amsfonts}
\usepackage{wasysym}
\usepackage{thmtools}
\usepackage{relsize}
\usepackage[mathscr]{eucal}
\usepackage[linktocpage]{hyperref}
\usepackage[sort,capitalize]{cleveref}
\usepackage{enumitem}			
\usepackage[dvipsnames]{xcolor}
\usepackage[msc-links, abbrev, non-sorted-cites]{amsrefs}
\usepackage{todonotes}
\usepackage{caption}

\usepackage{tikz}

\newcommand{\filledbox}{%
    \begin{tikzpicture}[baseline={(0,0)}]
        \draw[line width=1pt] (0,0) rectangle (.5em,.5em); 
    \end{tikzpicture}%
}

\makeatletter
\def\nonumberfootnote{\xdef\@thefnmark{}\@footnotetext}			
\makeatother

\definecolor{colorred}{HTML}{B00000}
\definecolor{colorgreen}{HTML}{258300}
\definecolor{colorblue}{HTML}{2e32fa}
\definecolor{coloryellow}{HTML}{cbbb1a}
\hypersetup{colorlinks=true, linkcolor=colorred, citecolor=colorgreen, urlcolor=coloryellow, pdfusetitle=true}

\numberwithin{equation}{section}

\newcommand{\rmc}{{\ensuremath{\mathrm{c}}}}
\newcommand{\rmd}{{\ensuremath{\mathrm{d}}}}
\newcommand{\rme}{{\ensuremath{\mathrm{e}}}}

\newcommand{\rmo}{{\ensuremath{\mathrm{o}}}}

\newcommand{\rmP}{{\ensuremath{\mathrm{P}}}}

\newcommand{\sfd}{{\ensuremath{\mathsf{d}}}}
\newcommand{\sfe}{{\ensuremath{\mathsf{e}}}}

\newcommand{\sfn}{{\ensuremath{\mathsf{n}}}}

\newcommand{\sfC}{{\ensuremath{\mathsf{C}}}}
\newcommand{\sfD}{{\ensuremath{\mathsf{D}}}}

\newcommand{\sfF}{{\ensuremath{\mathsf{F}}}}

\newcommand{\sfH}{{\ensuremath{\mathsf{H}}}}

\newcommand{\sfJ}{{\ensuremath{\mathsf{J}}}}

\newcommand{\sfL}{{\ensuremath{\mathsf{L}}}}

\newcommand{\sfN}{{\ensuremath{\mathsf{N}}}}

\newcommand{\sfT}{{\ensuremath{\mathsf{T}}}}

\newcommand{\scrD}{{\ensuremath{\mathscr{D}}}}

\newcommand{\bdpi}{{\ensuremath{\boldsymbol{\pi}}}}

\newcommand{\R}{\boldsymbol{\mathrm{R}}}						
\renewcommand{\d}{\,\mathrm{d}}				

\let\limsup\undefined

\let\div\undefined
\DeclareMathOperator*{\limsup}{limsup}		
\DeclareMathOperator*{\esssup}{esssup}		
\DeclareMathOperator{\supp}{spt}			
\DeclareMathOperator{\div}{div}				

\theoremstyle{definition}
\newtheorem{bump}{Bump}[section]

\theoremstyle{plain}
\newtheorem{theorem}[bump]{Theorem}
\newtheorem{proposition}[bump]{Proposition}
\newtheorem{definition}[bump]{Definition}
\newtheorem{lemma}[bump]{Lemma}
\newtheorem{corollary}[bump]{Corollary}

\newtheorem{conjecture}[bump]{Conjecture}

\theoremstyle{remark}
\newtheorem{remark}[bump]{Remark}
\newtheorem{example}[bump]{Example}

\crefname{theorem}{Theorem}{Theorems}
\crefname{proposition}{Proposition}{Propositions}
\crefname{definition}{Definition}{Definitions}
\crefname{lemma}{Lemma}{Lemmas}
\crefname{corollary}{Corollary}{Corollaries}
\crefname{hypothesis}{Hypothesis}{Hypotheses}
\crefname{remark}{Remark}{Remarks}
\crefname{example}{Example}{Examples}
\crefname{notation}{Notation}{Notations}
\crefname{idea}{Idea}{Idea}

\crefformat{section}{{§}#2#1#3}
\crefformat{subsection}{{§}#2#1#3}
\crefformat{subsubsection}{{§}#2#1#3}
\crefformat{appendix}{{§}#2#1#3}

\crefmultiformat{theorem}{Theorems #2#1#3}{ and #2#1#3}{, #2#1#3}{, and #2#1#3}
\crefmultiformat{proposition}{Propositions #2#1#3}{ and #2#1#3}{, #2#1#3}{, and #2#1#3}
\crefmultiformat{definition}{Definitions #2#1#3}{ and #2#1#3}{, #2#1#3}{, and #2#1#3}
\crefmultiformat{lemma}{Lemmas #2#1#3}{ and #2#1#3}{, #2#1#3}{, and #2#1#3}
\crefmultiformat{corollary}{Corollaries #2#1#3}{ and #2#1#3}{, #2#1#3}{, and #2#1#3}
\crefmultiformat{hypothesis}{Hypotheses #2#1#3}{ and #2#1#3}{, #2#1#3}{, and #2#1#3}
\crefmultiformat{remark}{Remarks #2#1#3}{ and #2#1#3}{, #2#1#3}{, and #2#1#3}
\crefmultiformat{example}{Examples #2#1#3}{ and #2#1#3}{, #2#1#3}{, and #2#1#3}
\crefmultiformat{notation}{Notations #2#1#3}{ and #2#1#3}{, #2#1#3}{, and #2#1#3}

\crefmultiformat{section}{{§§}#2#1#3}{ and #2#1#3}{, #2#1#3}{, and #2#1#3}
\crefmultiformat{subsection}{{§§}#2#1#3}{ and #2#1#3}{, #2#1#3}{, and #2#1#3}
\crefmultiformat{subsubsection}{{§§}#2#1#3}{ and #2#1#3}{, #2#1#3}{, and #2#1#3}

\crefrangeformat{equation}{#3\textcolor{black}{(}#1\textcolor{black}{)}#4 to #5\textcolor{black}{(}#2\textcolor{black}{)}#6}

\newcommand{\mms}{\mathsf{M}}				
\newcommand{\met}{\sfd}						

\newcommand{\Rmet}{g}						
\newcommand{\meas}{\mathfrak{m}}				

\newcommand{\Leb}{\mathscr{L}}				
\newcommand{\Haus}{\mathscr{H}}				
\newcommand{\vol}{\mathrm{vol}}				
\newcommand{\Prob}{\mathscr{P}}		        

\newcommand{\Id}{\mathrm{Id}}				

\newcommand{\TCD}{\mathsf{TCD}}

\newcommand{\CD}{\mathsf{CD}}
\newcommand{\LTCD}{\mathsf{LTCD}}
\newcommand{\RCD}{\mathsf{RCD}}

\newcommand{\comp}{\rmc}						

\newcommand{\pr}{\mathrm{pr}}				
\newcommand{\Ric}{\mathrm{Ric}}				

\newcommand{\Lip}{\mathrm{Lip}}				

\newcommand{\Dom}{\scrD}					
\DeclareMathOperator{\BOX}{\filledbox}

\DeclareMathOperator{\Ent}{Ent}				
\DeclareMathOperator{\Hess}{Hess}			
\newcommand{\eval}{\sfe}					

\newcommand{\push}{\sharp}					

\newcommand{\Len}{\mathrm{Len}}

\usepackage{blindtext}

\allowdisplaybreaks

\let\oldtocsection=\tocsection

\let\oldtocsubsection=\tocsubsection

\let\oldtocsubsubsection=\tocsubsubsection

\renewcommand{\tocsection}[2]{\hspace{0em}\oldtocsection{#1}{#2}}
\renewcommand{\tocsubsection}[2]{\hspace{1em}\oldtocsubsection{#1}{#2}}
\renewcommand{\tocsubsubsection}[2]{\hspace{2em}\oldtocsubsubsection{#1}{#2}}

\newcommand{\nocontentsline}[3]{}
\newcommand{\tocless}[2]{\bgroup\let\addcontentsline=\nocontentsline#1{#2}\egroup}

\newcommand{\mres}{\mathbin{\vrule height 1.6ex depth 0pt width 0.13ex\vrule height 0.13ex depth 0pt width 1.3ex}}

\newcommand{\MB}[1]{\textnormal{\textcolor{black}{#1}}}

\renewcommand{\u}{\mathit{u}}
\renewcommand{\v}{\mathit{v}}

\renewcommand{\q}{\mathfrak{q}}

\newcommand{\sing}{{\mathrm{sing}}}

\newcommand{\TCut}{\mathrm{TC}}

\newcommand{\hh}{\mathfrak{h}}

\allowdisplaybreaks

\makeatletter
\@namedef{subjclassname@2020}{\textup{2020} Mathematics Subject Classification}
\makeatother

\begin{document}

\title[New perspectives on the d'Alembertian]{New perspectives on the d'Alembertian from general relativity. An invitation}
\author{Mathias Braun}
\address{Institute of Mathematics, EPFL, 1015 Lausanne, Switzerland}
\email{\href{mailto:mathias.braun@epfl.ch}{mathias.braun@epfl.ch}}
\subjclass[2020]{Primary 
28A50, 
51K10;    
Secondary
35J92, 
49Q22, 
51F99, 
53C21, 
53C50, 
83C75. 
}
\keywords{Comparison theory; D'Alembertian; General relativity; Mean cur\-vature; Metric measure spacetime;  $p$-Laplacian; Ricci curvature.}
\thanks{Financial support by the EPFL through a Bernoulli Instructorship is gratefully acknowledged. This is an invited review for the conference ``Metric measure spaces, Ricci curvature, and optimal transport'' at Villa Monastero, Varenna, operated from September 23, 2024 to October 1, 2024. I sincerely thank Fabio Cavalletti, Matthias Erbar, Jan Maas, and Karl-Theodor Sturm for their  organization of this stimulating event and the great opportunity to compose this survey. I thank Tobias Beran, Matteo Calisti, Nicola Gigli, Robert McCann, Argam Ohanyan, Felix Rott, and Clemens Sämann for the fruitful joint work on \cite{beran-braun-calisti-gigli-mccann-ohanyan-rott-samann+-} and helpful comments. A substantial part of the surveyed works \cite{beran-braun-calisti-gigli-mccann-ohanyan-rott-samann+-,braun2024+} was conducted during my Postdoctoral Fellowships at the University of Toronto and my Fields Postdoctoral Fellowship at the Fields Institute for Research in Mathematical Sciences from 2022 to 2024. I sincerely thank these institutions and particularly my mentor Robert McCann for their generous support and the inspirational research conditions which made these excellent works possible. I also thank Nicola Gigli for his numerous comments which made me learn new facets even in my hobbyhorse of metric geometry. This overview was completed during the HIM Trimester Program ``Metric Analysis'' at the Hausdorff Research Institute for Mathematics, Bonn. I sincerely thank the institute for creating such a  stimulating research atmosphere.}

\begin{abstract} This survey has multiple objectives. First, we  motivate and review a new distributional notion of the d'Alembertian from mathematical relativity,  more precisely, a nonlinear $p$-version thereof, where $p$  is a nonzero number less than one. This operator comes from  natural Lagrangian actions introduced relatively recently. Unlike its classical linear yet hyperbolic counterpart, it is nonlinear yet has elliptic characteristics. Second, we describe recent comparison estimates for the $p$-d'Alembertian of Lorentz distance functions (notably a point or a spacelike hypersurface). Their new contribution implied by prior works  on optimal transport through spacetime is a control of the timelike cut locus. Third, we illustrate exact representation formulas for these $p$-d'Alembertians  employing methods from convex geometry. Fourth, several applications and open problems are presented.
\end{abstract}

\maketitle

\thispagestyle{empty}

\tableofcontents

\addtocontents{toc}{\protect\setcounter{tocdepth}{2}}

\section{Introduction}\label{Ch:Intro}

This paper surveys the two recent contributions by Beran--Braun--Calisti--Gigli--McCann--Ohanyan--Rott--Sämann \cite{beran-braun-calisti-gigli-mccann-ohanyan-rott-samann+-} and Braun \cite{braun2024+} to the geometric analysis of the d'Alem\-bertian and its comparison theory in mathematical relativity from \MB{the perspective of a new family of Lagrangians}. 

In Riemannian geometry, the Laplacian and its relation to the Ricci curvature --- especially expressed by its numerous  comparison results --- are well-acknowledged across geometry, analysis, and probability, cf.~e.g.~the monographs of Cheeger--Ebin \cite{cheeger-ebin1975}, Chavel \cite{chavel1984}, Grigor'yan \cite{grigoryan2009}, Bakry--Gentil--Ledoux \cite{bakry-gentil-ledoux2014}, and Wang \cite{wang2014-analysis}. The relevance of comparison methods in Lorentzian geometry, in particular with regards to the groundbreaking singularity theorems of Penrose \cite{penrose1965}, Hawking \cite{hawking1967}, and Hawking--Penrose \cite{hawking-penrose1970}, was already well-known to experts in Lorentzian geometry a long time ago. Indeed, as is  standard to infer, this link is certified by (closely related) geometric equations by Riccati, Raychaudhuri, and Bochner\footnote{In Riemannian signature, the Bochner identity lead to the first synthetic approach to lower bounds on the Ricci curvature by Bakry--\smash{É}mery \cite{bakry-emery1985-diffusions} in terms of their $\Gamma$-calculus. They were also the first pointing out the interrelation between Laplacians with drift term, weighted reference measures, and lower bounds on a modified Ricci tensor (nowadays well-known as \emph{Bakry--Émery--Ricci curvature}).}. 
It is somewhat surprising that a systematic comparison theory in Lorentzian signature --- apart from an earlier result of Eschenburg \cite{eschenburg1988} detailed below --- was initiated only around the turn to the 21st century by Ehrlich--Jung--Kim \cite{ehrlich-jung-kim1998} and  Ehrlich--Sánchez \cite{ehrlich-sanchez2000} and later extended by  Treude \cite{treude2011}, Treude--Grant \cite{treude-grant2013}, Graf \cite{graf2016}, and Lu--Minguzzi--Ohta \cite{lu-minguzzi-ohta2022-range}.

The Lorentzian analog of the Riemannian Laplacian is the well-known \emph{d'Alem\-bertian} $\Box$. Besides its indicated link to the Ricci tensor, its geometric significance derives from the fact that, roughly speaking, the d'Alembertian of the Lorentz distance function to a suitable  hypersurface $\Sigma$ describes the mean curvature of $\Sigma$. Moreover, it occurs in the splitting theorems of Eschenburg \cite{eschenburg1988}, Galloway \cite{galloway1989-splitting}, and Newman \cite{newman1990} in two prominent ways: the d'Alembert comparison theorem of Eschenburg \cite{eschenburg1988} and the maximum principle. 

The Lorentzian signature of the metric tensor of a spacetime  complicates the applicability of these methods. Indeed, the d'Alembertian is a \emph{hyperbolic} operator, unlike the elliptic Laplacian from Riemannian manifolds. All its prior  comparison theorems \cite{eschenburg1988,treude2011,treude-grant2013} entailed no information on the d'Alembertian within the timelike cut locus in question. Since a global control across this singular set is required for the splitting theorem, the works  \cite{eschenburg1988,galloway1989-splitting,newman1990} formulated the necessary estimates in the barrier sense of Calabi \cite{calabi1958}. A distributional version as in  Cheeger--Gromoll's Riemannian splitting theorem  \cite{cheeger-gromoll1972} covering the singular timelike cut locus would talk better to the  integration by parts formula defining $\Box$, yet was unavailable in the literature thus far. The barrier formulation of \cite{calabi1958} comes from the maximum principle. In Lorentzian signature, thus far maximum principle  tools have only been applied  after restricting $\Box$ to certain spacelike hypersurfaces, where it becomes the  Riemannian Laplacian. This extrinsic and nonglobal procedure requires sophisticated smoothness  \emph{a priori} (such as a one-sided Hessian bound on the Lorentz distance function in question, cf.~e.g.~Andersson--Galloway--Howard \cite{andersson-galloway-howard1998}). However, smoothness is a conceptual consequence of ellipticity (the backbone of the maximum principle), not a hypothesis. 

The recent work of Beran--Braun--Calisti--Gigli--McCann--Ohanyan--Rott--Sämann \cite{beran-braun-calisti-gigli-mccann-ohanyan-rott-samann+-} and its sequel by Braun \cite{braun2024+} surveyed in this paper (as well as a sequel by Braun--Gigli--McCann--Ohanyan--Sämann \cite{braun-gigli-mccann-ohanyan-samann+} reviewed by McCann \cite{mccann+})  tackle   these issues successfully. In a nutshell, the proposal is to replace the linear yet hyperbolic d'Alem\-bertian $\Box$ by the \emph{$p$-d'Alembertian}, formally given by 
\begin{align}\label{Eq:Operatorrr}
\Box_p := \div\!\big[\big\vert\rmd \cdot\big\vert_*^{p-2}\,\nabla\cdot\big].
\end{align}
Here $p$ is a nonzero number less than one. The term  \eqref{Eq:Operatorrr} reflects the $p$-Laplacian, which is well-studied in nonlinear PDEs. In our realm of mathematical relativity, the somewhat unusual range of $p$ comes  from a natural flock of Lagran\-gians on the tangent bundle pioneered  by McCann \cite{mccann2020}, Minguzzi \cite{minguzzi2019-causality}, and Mondino--Suhr \cite{mondino-suhr2022}; the argument inside the divergence in  \eqref{Eq:Operatorrr} is  the induced Legendre transform of the exterior differential $\rmd$. Defining Laplace-type operators in this way is in fact usual business in Hamiltonian geometry, cf.~e.g.~Agrachev--Gamkrelidze \cite{agrachev-gamkrelidze1997} and Agrachev \cite{agrachev2007,agrachev2008}; in this case, comparison theorems have been established by Ohta \cite{ohta2014}. (Note that the Lagrangians of  \cite{mccann2020,minguzzi2019-causality,mondino-suhr2022} are not covered by this theory,  as they degenerate  outside light cones.) As clear as this indicative connection might be to differential geometers or  experts in nonlinear PDEs  a posteriori, the first occurrence of the $p$-d'Alembertian in mathematical relativity we know is the quite   recent work of Mondino--Suhr \cite{mondino-suhr2022}.

As argued in \cite{beran-braun-calisti-gigli-mccann-ohanyan-rott-samann+-,braun2024+,braun-gigli-mccann-ohanyan-samann+},  further good reasons for a systematic study of \eqref{Eq:Operatorrr} seem to exist. First, in all geometrically relevant problems outlined above, since Lorentz distance functions have unit slope on the relevant sets, their $p$-d'Alem\-bertian coincides with the classical d'Alembertian. Thus, $\Box_p$ inherits the geometric significance of $\Box$. Second, as first realized by Beran--Braun--Calisti--Gigli--McCann--Ohanyan--Rott--Sämann \cite{beran-braun-calisti-gigli-mccann-ohanyan-rott-samann+-} and made concrete later by Braun--Gigli--McCann--Ohanyan--Sämann \cite{braun-gigli-mccann-ohanyan-samann+}, $\Box_p$ is a nonlinear yet \emph{elliptic} operator. This holds since  $\Box_p$ is the variational derivative of a convex energy functional \cite{beran-braun-calisti-gigli-mccann-ohanyan-rott-samann+-}. \MB{(Its convexity, which a priori holds on the cone of causal functions, easily extends to larger classes of functions if appropriate infinity conventions --- depending on the sign of $p$ --- are imposed.)} Alternatively \cite{braun-gigli-mccann-ohanyan-samann+}, in local coordinates one may formally write
\begin{align*}
\Box_p\u = \big\vert\rmd\u\big\vert_*^2\,\Big[(2-p)\frac{\nabla^i\u\,\nabla^j\u}{\big\vert\rmd\u\big\vert_*^2} - g^{ij}\Big]\,\textcolor{black}{\nabla_i\nabla_j\u} + \textnormal{lower order terms} 
\end{align*}
in Wald's ``abstract index notation'' \cite{wald1984}. By the $+,-,\dots,-$ signature of the Lo\-rentzian metric  $g$, the highest order coefficient matrix of $\smash{\Box_p\u}$ has eigenvalues close to $1-p,1,\dots,1>0$ in appropriate coordinates. Based on this simple idea  Braun--Gigli--McCann--Ohanyan--Sämann \cite{braun-gigli-mccann-ohanyan-samann+} give a drastically simpler proof of the splitting theorem outlined above, which is reviewed by McCann \cite{mccann2020}.

The main results we survey in this paper are the following. 
\begin{itemize}
\item Employing methods from optimal transport and metric geometry,  Beran--Braun--Calisti--Gigli--McCann--Ohanyan--Rott--Sämann \cite{beran-braun-calisti-gigli-mccann-ohanyan-rott-samann+-}  proved comparison theorems for the $p$-d'Alem\-bertian of a Lorentz distance function from a point in the distributional spirit of Cheeger--Gromoll \cite{cheeger-gromoll1972}. Unlike preceding works, \textit{their result holds across the timelike cut locus in question}. On spacetimes, this essentially comes from the good control of optimal transport over singular sets proven  by McCann \cite{mccann2020}. From these comparison results, \cite{beran-braun-calisti-gigli-mccann-ohanyan-rott-samann+-} then derived the \textit{existence of the distributional $p$-d'Alembertian} of such functions (and appropriate powers thereof as well as general Kantorovich potentials)  using the Riesz--Markov--Kakutani representation theorem. The proof strategy and the  distributional definition of the $p$-d'Alembertian proposed in \cite{beran-braun-calisti-gigli-mccann-ohanyan-rott-samann+-}  are inspired by and follow preceding work of Gigli \cite{gigli2015} on the distributional Laplacian of metric measure spaces with synthetic Ricci curvature bounds à la Sturm \cite{sturm2006-i,sturm2006-ii} and Lott--Villani \cite{lott-villani2009}.
\item Complementary to this abstract view on the $p$-d'Alembertian, Braun \cite{braun2024+} established \emph{exact representation formulas} for the $p$-d'Alembertian of Lorentz distance functions from appropriate sets $\Sigma$ (and suitable  powers thereof). From these formulas, generalizations of the comparison results of Treude \cite{treude2011}, Treude--Grant \cite{treude-grant2013}, Graf \cite{graf2016}, Lu--Minguzzi--Ohta \cite{lu-minguzzi-ohta2022-range}, and  Beran--Braun--Calisti--Gigli--McCann--Ohanyan--Rott--Sämann \cite{beran-braun-calisti-gigli-mccann-ohanyan-rott-samann+-} can be read off almost straightforwardly by elementary concavity estimates.  The approach to these representation formulas is inspired by and follows prior work of Cavalletti--Mondino \cite{cavalletti-mondino2020-new} on metric measure spaces with synthetic Ricci curvature bounds. A crucial ingredient of \cite{braun2024+} is the Lorentzian localization technique pioneered by Cavalletti--Mondino \cite{cavalletti-mondino2020}. Based on these formulas and  prior work of Ketterer \cite{ketterer2020-heintze-karcher}, \cite{braun2024+} then generalized  a \textit{volume comparison inequality} due to Treude--Grant \cite{treude-grant2013} and Graf--Sormani \cite{graf-sormani2022}, which was originally inspired by the famous Riemannian Heintze--Karcher inequality \cite{heintze-karcher1978}. Finally, in \cite{braun2024+} this estimate was applied to the quite new concept of \emph{volume singularities}  proposed by García-Heveling \cite{garcia-heveling2023-volume} which was inspired by Treude--Grant \cite{treude-grant2013} \MB{and Graf--Sormani \cite{graf-sormani2022}} and in turn Heintze--Karcher \cite{heintze-karcher1978}.
\end{itemize}

We finish our introduction with some remarks about our setting. The reviewed results of \cite{beran-braun-calisti-gigli-mccann-ohanyan-rott-samann+-,braun2024+} hold in the abstract setting of metric measure spacetimes with synthetic timelike Ricci curvature bounds. The latter notion was introduced by Cavalletti--Mondino \cite{cavalletti-mondino2020}; the term ``metric measure spacetime'',  suggested by McCann \cite{mccann2023-null}, broadly refers to the abstract generalization of spacetimes in terms of metric measure geometry, pioneered by Kunzinger--Sämann \cite{kunzinger-samann2018} and further studied in sequella of Müller \cite{muller2022}, Minguzzi--Suhr \cite{minguzzi-suhr2022}, McCann \cite{mccann2023-null}, Braun--McCann \cite{braun-mccann2023}, Beran--Braun--Calisti--Gigli--McCann--Ohanyan--Rott--Sämann \cite{beran-braun-calisti-gigli-mccann-ohanyan-rott-samann+-}, and Bykov--Minguzzi--Suhr \cite{bykov-minguzzi-suhr2024+}. (See Sormani--Vega \cite{sormani-vega2016} for an alternative.) The surveyed works are  based on optimal transport and metric geometry, theories Lorentzian geometers and mathematical physicists may feel more comfortable with in a ``traditional'' realm. On the other hand, metric geometers reading this special issue who know the underlying  ideas (notably those of Gigli \cite{gigli2015} and Cavalletti--Mondino \cite{cavalletti-mondino2020-new}) may be interested in the Lorentzian background of their recent transplantation into structures with synthetic timelike Ricci curvature bounds. To accommodate both sides, we have decided to trim down our presentation from the abstract settings of \cite{beran-braun-calisti-gigli-mccann-ohanyan-rott-samann+-,braun2024+} to classical spacetimes, outsourcing less classical structures (and the numerous motivations to study them) to the surveys of Steinbauer \cite{steinbauer2023},  Cavalletti--Mondino \cite{cavalletti-mondino2022-review},   and Sämann \cite{samann2024+}. Many of the surveyed results are new already in this smooth case. On the other hand, our chosen reduction only partly reflects the enormous technical efforts  invested in \cite{beran-braun-calisti-gigli-mccann-ohanyan-rott-samann+-,braun2024+}. Several key contributions are omitted entirely in our survey. Examples from \cite{beran-braun-calisti-gigli-mccann-ohanyan-rott-samann+-} are
\begin{itemize}
\item the first-order Sobolev-type calculus  (developed from scratch!) which makes the machinery we outline below work in a genuinely nonsmooth setting,
\item abstract exponentiation by a Brenier--McCann theorem,
\item a nonsmooth ``converse'' Hawking--King--McCarthy theorem, and
\item a dynamical picture of the Lorentz--Wasserstein ``spacetime'' of mass distributions over a metric measure spacetime;
\end{itemize}
examples from \cite{braun2024+} are
\begin{itemize}
\item a synthetic Bochner-type inequality and
\item a genuinely nonsmooth notion of mean curvature.
\end{itemize}
Instead, we  focus on the most fundamental topics that strictly relate to the new  $p$-d'Alembertian \eqref{Eq:Operatorrr}. We frame the \emph{fil rouge} leading to its comparison theory and existence results of \cite{beran-braun-calisti-gigli-mccann-ohanyan-rott-samann+-,braun2024+} in the traditional spacetime setting, instead of introducing an entire set of nonsmooth techniques and terminology. Along the way, we try  to balance mathematical rigor with rather informal heuristics and several simplifications. All proofs, if any, will only be sketched.

Our hope --- reflecting how this survey should be understood, cf.~its title --- is the flow of (in fact very natural and simple)  ideas described in this  paper   motivates experts and nonexperts in synthetic  Lorentzian geometry to read the recent  works of Beran--Braun--Calisti--Gigli--McCann--Ohanyan--Rott--Sämann \cite{beran-braun-calisti-gigli-mccann-ohanyan-rott-samann+-} and Braun \cite{braun2024+} in detail.  Furthermore, it may serve as a blueprint to read along while inspecting its nonsmooth analogs  in \cite{beran-braun-calisti-gigli-mccann-ohanyan-rott-samann+-,braun2024+}. As a black box, the reader may want to keep in mind that all central ingredients from spacetime geometry and the Lorentzian optimal transport problem  pointed out throughout this survey admit generalizations to abstract metric measure spacetimes. An overview about differential calculus in metric measure spaces (or, more broadly, the nonsmooth splitting theorem  of Gigli \cite{gigli2013}) that inspired \cite{beran-braun-calisti-gigli-mccann-ohanyan-rott-samann+-} can be found in Gigli's survey \cite{gigli2014-overview}. Lastly, on a broader scale we hope this survey helps nonexperts navigate through the rapidly growing recent literature on Lorentzian optimal transport.

\section{Spacetime geometry}

We outline some fundamentals of the standing setting for our  survey: (smooth) spacetimes. For details, we refer to the classics    \cite{hawking-ellis1973,oneill1983,wald1984,beem-ehrlich-easley1996}. Generalizations are indicated further below, together with relevant literature. In the current part, we focus on a concise presentation of mathematical background, also attempting to clarify which quantity depends on which  given  datum (notably metric tensor or reference measure). 
We recommend the surveys \cite{steinbauer2023,samann2024+,chrusciel2011,minguzzi2019-causality,cavalletti-mondino2022-review,mccann+} for more physical interpretations of the introduced geometric objects. 

This part does not introduce new ideas and can be skipped by readers familiar with the basics of Lorentzian geometry.

Throughout this survey, every  topological manifold will be tacitly assumed to be connected and smooth (unless it is discrete) and have empty boundary, unless explicitly stated otherwise. By convention, the topology of each appearing manifold is   Hausdorff and second countable. 

Let $\mms$ be such a manifold. We  assume its dimension $\dim \mms$ is at least two. For simplicity, cf.~\cref{Re:Noncpt}, we suppose $\mms$ is noncompact from the outset.



\begin{remark}[Relevant notions from Riemannian geometry]\label{Re:Riemannian} A theorem of Nomizu--Ozeki \cite{nomizu-ozeki1961} yields the existence of a complete smooth Riemannian metric $r$ on $\mms$ which induces its topology. It is fixed from now on. Let $\nabla^r$ and $\met_r$ designate the Levi-Civita connection and the Riemannian distance function associated with $r$, respectively. For $v\in T\mms$, we also write $\smash{\vert v\vert_r := \sqrt{r(v,v)}}$. As customary, we call a function $f$ defined on an open  subset $U$ of $\mms$ \emph{locally Lipschitz continuous} if for every $z\in\mms$, there are  $L>0$ and a neighborhood $V$ of $z$ with $\vert f(x) - f(y)\vert\leq L\,\met_r(x,y)$ for every $x,y\in U\cap V$.  The local Lipschitz continuity of a curve $\gamma$ defined on an interval $I$ is defined analogously. As this property (as well as all other relevant subsequent claims   involving $r$) is local, it does not depend on the  choice of $r$. \hfill$\blacksquare$
\end{remark}

\subsection{Lorentzian metrics and adjacent geometric objects}\label{Sub:Lorentz metric}

\begin{definition}[Lorentzian metric]\label{Def:Lormet} A \emph{Lorentzian metric} refers to a smooth symmetric section $g$ of $\smash{T^*\mms^{\otimes 2}}$ with constant signature $+,-,\dots,-$ throughout $\mms$. 
\end{definition}


Throughout our discussion, we fix a Lorentzian metric $g$ on $\mms$.

Given $v\in T\mms$, we will write $\vert v\vert^2 := \langle v,v\rangle:= g(v,v)$. If  $\smash{\vert v\vert^2}$ is nonnegative, we also define $\smash{\vert v\vert:= \sqrt{\vert v\vert^2}}$.  We  employ analogous notations $\smash{\vert\cdot\vert_*^2}$ and $\smash{\vert\cdot\vert_*}$ as above for the cometric $\smash{g^*}$ induced by $g$, defined by  $\smash{g^*(\zeta,\zeta) := \vert \zeta^\sharp\vert^2}$ for every $\zeta\in T^*\mms$;   
 the nondegeneracy of $g$ at every tangent space induces the \emph{musical isomorphism} $\sharp\colon T^*\mms \to T\mms$ given by the formula $\smash{\zeta(v) = \langle\zeta^\sharp,v\rangle}$ for every $\zeta\in T^*\mms$ and every $v\in T\mms$. The identification of cotangent and tangent spaces in more general cases (e.g.~Finsler spacetimes) in terms of convex analysis is more elaborate. A related yet slightly different identification  paradigm  is the backbone of the nonlinear approach to the d'Alembertian surveyed in this article;  see \cref{Sub:Lagr} for details.


Analogously to Riemannian manifolds, the Lorentzian metric $g$ induces several  geometric quantities. We only recall those which are most relevant for us.
\begin{itemize}
\item \textbf{Gradient.} For a smooth function $\u$ on $\mms$, let $\rmd\u$ be its usual differential. Recall  that $\rmd\u$ only depends on the differentiable structure of $\mms$. Then the   \emph{gradient} $\nabla \u$ is the smooth section of $T\mms$ given by $\smash{\nabla \u := (\rmd\u)^\sharp}$. 
Thus, the dependence of $\nabla\u$ on $g$ enters through the musical isomorphism $\sharp$. 
\item \textbf{Hessian.} The \emph{Hessian} of a function $\u$ as above is the smooth section of $\smash{T^*\mms^{\otimes 2}}$ given by $\Hess\u := \nabla^2\u$.
\item \textbf{Ricci curvature.} The \emph{Ricci tensor} $\Ric$ is the smooth section of $\smash{T^*\mms^{\otimes 2}}$ coming from ``tracing'' --- in the customary way --- the Riemann tensor $\mathrm{Rm}$, which acts on smooth vector fields $X$, $Y$, and $Z$ on $\mms$ by the formula
\begin{align*}
\mathrm{Rm}(X,Y)Z := \nabla_X\nabla_YZ -\nabla_Y\nabla_X Z- \nabla_{[X,Y]}Z;
\end{align*}
here, ``tracing'' is understood with respect to the first slot. Moreover, $\nabla$ is the Levi-Civita connection induced by $g$ (via Koszul's formula) while $[\cdot,\cdot]$ is the usual Lie bracket.
\end{itemize}

\subsection{The ``measure''. Reference measures and divergence}\label{Sub:Refmeas} 

A Borel measure $\meas$ on $\mms$ is called \emph{smooth} if it is absolutely continuous with respect to the $\dim\mms$-dimensional Lebesgue measure $\smash{\Leb^{\dim\mms}}$ with smooth and positive density  on every coordinate chart\footnote{We tacitly use the customary identifications, albeit the two measures in question technically live on different spaces.}. In particular, $\meas$ and $\smash{\Leb^{\dim\mms}}$ are mutually absolutely continuous, meaning they share the same null sets.   This allows us to say in the sequel that a Borel subset $E$ of $\mms$ has measure zero if it has $\meas$-measure zero for some (hence every) smooth measure $\meas$ on $\mms$.

Smooth measures considered in this survey are the following.
\begin{itemize}
\item \textbf{Volume measure.} The most prominent example is the \emph{volume measure} $\vol$ induced by  $g$: it has density $\smash{\sqrt{\vert\!\det g\vert}}$ with respect to $\smash{\Leb^{\dim\mms}}$. 
\item \textbf{Weighted measures.}  A more general class of reference measures take the form $\rme^V\,\vol$, where $V$ is a  smooth  function on $\mms$.  In fact, by a partition of unity argument one sees every smooth measure on $\mms$ is of this form.
\end{itemize}


Any  smooth measure $\meas$ on $\mms$ induces a  \emph{divergence} operator $\div_\meas$. By Weyl's lemma, it is uniquely determined by the Gauß--Green formula
\begin{align}\label{Eq:div def}
\int_\mms \rmd \varphi(X)\d\meas = -\int_\mms \varphi\div_\meas X\d\meas,
\end{align}
valid for every $\smash{\varphi\in C_\comp^\infty(\mms)}$. Here we assume the vector field $X$ to be smooth to simplify the presentation; however, the weak formulation \eqref{Eq:div def} and the smoothness of $\meas$ clarify  $X$ may well have lower regularity  (e.g.~local Lipschitz continuity) to  possess a well-defined divergence with respect to $\meas$. By  \eqref{Eq:div def}, $\div_\meas$ only depends on the choice of the reference measure $\meas$, but not on the    Lorentzian metric $g$ --- except, of course, when $\meas$ itself does.

If $\meas$ is the volume measure $\vol$, we write $\div$ for the associated divergence. 

\subsection{Causality theory}\label{Sub:Caus thry}  

\subsubsection{Causal characters} We call a tangent vector $v\in T\mms\setminus\{0\}$ 
\begin{itemize}
\item \emph{timelike} if $\smash{\vert v\vert^2 >0}$, \item \emph{lightlike} if $\smash{\vert v\vert^2 =0}$,
\item  \emph{causal} if $\smash{\vert v\vert^2\geq 0}$, and 
\item \emph{spacelike} if $\smash{\vert v\vert^2 <0}$. 
\end{itemize}
The zero vector is spacelike by convention.  These adjectives describe the so-called \emph{causal character}  of the tangent vector in question. 

Physically, timelike and lightlike directions correspond to directions of movement of massive observers and light rays through spacetime, respectively.

The causal character of a smooth vector field $X$ on $\mms$ is defined by requiring all its point evaluations to be timelike, lightlike, causal, or spacelike,  respectively. If $X$ is  timelike, lightlike, or spacelike (meaning the quantity $\smash{\vert X\vert^2}$ is always strictly positive, zero, or strictly negative throughout $\mms$), we say it has a definite causal character. Correspondingly, a curve $\gamma$ through $\mms$ is called timelike, light\-like, causal, or spacelike if it is locally Lipschitz continuous and $\smash{\Leb^1}$-a.a.~of its tangent vectors  have the respective causal character. 


\subsubsection{Time orientation} Clearly, the causal character of a tangent vector $v\in T\mms$ is unchanged by flipping the sign of $v$. This necessitates a criterion telling us when $v$ points into the ``future'' or the ``past''. This is done by fixing a time orientation, i.e.~a continuous timelike vector field $X$ on $\mms$. First, note that since the zero vector is spacelike, the set of causal vectors at each tangent space has more than one connected component, and their closures intersect precisely at zero (see e.g. Minguzzi \cite{minguzzi2015-light}). For $x\in \mms$, we then call $v\in T_x\mms$ \emph{future-directed} if it lies in the same connected component as $X_x$ and \emph{past-directed} otherwise. Simple examples (mimicking the failure of the  Möbius strip to be orientable in the Riemannian sense)  show not every Lorentzian metric is time orientable. Unless explicitly stated otherwise, all vectors, vector fields, and curves are tacitly assumed to be future-directed. 

\begin{definition}[Spacetime] A \emph{spacetime} $(\mms,g)$  refers to a smooth manifold $\mms$ as above endowed with a  Lorentzian metric $g$ and a time orientation. 
\end{definition}

\begin{example}[Model spaces {\cite{treude-grant2013}}]\label{Ex:Model} As in Riemannian geometry, the standard way to construct spacetimes with constant  curvature (where many quantitative geometric inequalities are sharp) are \emph{warped products}. They are set up as follows. Let $I$ be an open interval in $\R$. Let $(\sfN, r)$ be a complete Riemannian manifold. Consider the product $\mms := I\times \sfN$. Given a smooth and positive function $f$ on $I$, we consider the Lorentzian metric $\Rmet := \rmd t^2 - f(t)\,r$ on $\mms$.  (Completeness of $(\sfN,r)$ implies global hyperbolicity of $(\mms,g)$ in the sense of \cref{Sub:Glob hyp} below.) Choosing a time orientation  making the vector field $\partial/\partial t$ future-directed, this  turns $\mms$ and $g$ into a spacetime.

The model classes  for constantly negative, zero, and positive sectional (hence Ricci) curvature are called \emph{de Sitter}, \emph{Minkowski}, and \emph{anti-de Sitter spacetimes}. They can be constructed explicitly from natural choices of $(\sfN,r)$ and $f$ \cite{treude-grant2013}*{§4.2}. \hfill$\blacksquare$ 
\end{example}

\subsubsection{Chronology and causality} We define two relations on $\mms$ as follows.
\begin{itemize}
\item \textbf{Chronology.} Given $x,y\in\mms$, we write $x\ll y$ if \MB{there exists a timelike curve connecting} $x$ and $y$.
\item \textbf{Causality.} Given $x,y\in\mms$, we write $x\leq y$ if \MB{there exists a causal curve connecting} $x$ and  $y$. 
\end{itemize}
 Clearly, the relations $\ll$ and $\leq$ are transitive, $\leq$ is reflexive, and $\ll$ is included in $\leq$. Moreover (yet less obviously), the \emph{push-up principle} holds, meaning if a triple $x,y,z\in\mms$ obeys $x\ll y \leq z$ or $x \leq y\ll z$, then $x\ll z$. 

As customary, the relations $\ll$ and $\leq$ can be regarded as subsets $I$ and $J$ of $\mms^2$, respectively. The push-up principle is then shortly written $(I\circ J) \cup (J\circ I) \subset I$. Here and in the sequel, all sets tagged with ``$I$'' will be open. Although we only consider  smooth Lorentzian metrics for now, we note if $g$ is not locally Lipschitz continuous, these properties may fail in general, as pointed out by Chru\'sciel--Grant \cite{chrusciel-grant2012} and  Grant--Kunzinger--Sämann--Steinbauer \cite{grant-kunzinger-samann-steinbauer2020}.

\subsubsection{Some sets}\label{Sub:Some}
Given  $x,y\in\mms$, we  define 
\begin{itemize}
\item the \emph{chronological future} of $x$ as 
\begin{align*}
I^+(x) := \pr_2\big[I \cap (\{x\}\times\mms)\big],
\end{align*}
\item the \emph{causal future} of $x$ as 
\begin{align*}
J^+(x) := \pr_2\big[J \cap (\{x\}\times\mms)\big],
\end{align*}
\item the \emph{chronological past} of $y$ as 
\begin{align*}
I^-(y) := \pr_1\big[I \cap (\mms\times \{y\})\big],
\end{align*}
\item the \emph{causal past} of $y$ as
\begin{align*}
J^-(y) := \pr_1\big[J \cap (\mms\times \{y\})\big].
\end{align*}
\end{itemize}
Here and in the sequel, $\pr_\imath \colon \mms^2\to \mms$ means the projection onto the $\imath$-th coordinate, where $\imath \in\{1,2\}$. Then the \emph{chronological diamond} $I(x,y)$ and the \emph{causal diamond} $J(x,y)$ are defined by $\smash{I^+(x) \cap I^-(y)}$ and $\smash{J^+(x)\cap J^-(y)}$, respectively. For subsets $X$ and $Y$ of $\mms$, $I(X,Y)$ is the union of all chronological diamonds $I(x,y)$, where $x\in X$ and $y\in Y$. We define $J(X,Y)$ analogously. This set is convex with respect to the relation $\leq$, i.e.~for every $x,y\in J(X,Y)$, the causal diamond $J(x,y)$ forms a subset of $J(X,Y)$. In particular, every causal curve with endpoints in $J(X,Y)$ is contained in $J(X,Y)$. Analogous  statements hold by replacing all occurrences of ``$J$'', ``$\leq$'', and ``causal'' by ``$I$'', ``$\ll$'', and ``timelike'', respectively. If $X$ and $Y$ are compact, the sets $I(X,Y)$ and $J(X,Y)$ are also called \emph{chronological emerald} and \emph{causal emerald} spanned by $X$ and $Y$, respectively.


\subsubsection{Causality condition} There are various natural causality conditions (i.e.~properties of the relation $\leq$) one can impose on the spacetime $(\mms,g)$. We refer to the reviews of Chru\'sciel \cite{chrusciel2011} or Minguzzi \cite{minguzzi2019-causality} for  details. The metatheorem listing their implications, e.g.~\cite{minguzzi2019-causality}*{§2.11}, is often called \emph{causal ladder}\footnote{As I realized based on comments of Argam Ohanyan and Clemens Sämann, it is somewhat ambiguous in the literature which properties one includes in the causal ladder. In this survey, we follow the list of Minguzzi \cite{minguzzi2019-causality} which collects properties strictly relating to causality, excluding chronological notions. Consequently, what we deem as strongest or weakest property on the causal ladder should be  interpreted relative to the causal ladder from  \cite{minguzzi2019-causality}.}; the stronger a property, the higher its step is located therein.  Its lowest  step is antisymmetry of $\leq$, often called \emph{causality} of $(\mms,g)$. That is, if $x,y\in\mms$ are such that $x\leq y$ and $y\leq x$ simultaneously, then $x=y$. This excludes the occurrence of causal loops, i.e.~nonconstant causal curves starting and ending in the same point. 

\begin{remark}[Causality forces noncompactness]\label{Re:Noncpt} If $\mms$ is compact and, as assumed, has empty  boundary, a simple covering argument shows the existence of a closed timelike curve. Thus, in our case compact spacetimes are never causal. \hfill$\blacksquare$
\end{remark}

\subsubsection{Global hyperbolicity}\label{Sub:Glob hyp}  The key concept of \emph{global hyperbolicity} was introduced by Leray \cite{leray1953} in the context of hyperbolic PDEs. It has entered mathematical relativity through the proof of global well-posedness of the Einstein equations by Choquet-Bruhat \cite{foures-bruhat1952} and Choquet-Bruhat--Geroch \cite{choquet-bruhat-geroch1969}, the singularity theorems of Penrose \cite{penrose1965} and  Hawking \cite{hawking1967} (see also \cref{Sub:Volume sing}), and Geroch's topological splitting theorem \cite{geroch1970-domain}. In contrast to causality of $(\mms,g)$ outlined above, global hyperbolicity lies on the highest step of the causal ladder. In later sections of this survey, it will be part of our standing hypotheses.

\begin{definition}[Global hyperbolicity  \cite{minguzzi2013-convexity}]\label{Def:Glob hyp} The spacetime $(\mms,g)$  is termed \emph{globally hyperbolic} if the following two conditions hold simultaneously.
\begin{enumerate}[label=\textnormal{\alph*\textcolor{black}{.}}]
\item\label{La:causa} The order $\leq$ is antisymmetric and closed.
\item\label{La:Diamonds} Compactness is stable by taking  causally convex hulls. That is, for every compact subset $C$ of $\mms$, the causal emerald $J(C,C)$ is compact.
\end{enumerate}
\end{definition}


The above is not  the classical definition \cite{hawking-ellis1973}. The notion of global hyperbolicity from \cref{Def:Glob hyp} we use was suggested for general topological ordered spaces by Minguzzi \cite{minguzzi2013-convexity}. There are several classical \cite{geroch1970} and newer \cite{bernal-sanchez2007,minguzzi2009-characterization,minguzzi2013-convexity,burtscher-garcia2024-eyes} variations of \cref{Def:Glob hyp} of varying strength which are all equivalent in the spacetime case. (Among these, \cref{Def:Glob hyp} is comparably strong per se.) We do not review them in detail to make this survey more focused. The reader wishing to have a clean  definition of global hyperbolicity may concentrate on the high (and physically relevant)  generality covered by the subsequent result of Hounnonkpe--Minguzzi \cite{hounnonkpe-minguzzi2019}. It yields sufficient conditions under which the condition \ref{La:causa} from \cref{Def:Glob hyp} can be discarded entirely and \ref{La:Diamonds} can be weakened.

\begin{theorem}[Global hyperbolicity with less conditions \cite{hounnonkpe-minguzzi2019}] Noncompactness of $\mms$ yields the following. If $\dim \mms$ is at least three then $(\mms,g)$ is globally hyperbolic if and only if for all points $x,y\in\mms$,  the causal diamond $J(x,y)$ is compact.
\end{theorem}

\subsection{The ``metric''. Time separation function and geodesics}\label{Sub:Time sep} Every  Riemannian manifold can be turned into a metric space by minimizing the length functional induced from the  Riemannian metric, cf.~\eqref{Eq:Riemanniandistance}. The induced topology is compatible with the one inherited from the  space  in question being a topological manifold.

The parallel ``metric'' information of the spacetime $(\mms,g)$ is encoded by the  \emph{time separation function} $l\colon \mms^2\to \R_+\cup\{-\infty,\infty\}$. It is  defined by
\begin{align}\label{Eq:l function}
l(x,y) := \sup\!\Big\lbrace\!\int_0^1 \vert\dot\gamma_t\vert\d t : \gamma\textnormal{ causal curve with }\gamma_0= x\textnormal{ and }\gamma_1=y\Big\rbrace;
\end{align}
here, we use the convention $\sup\emptyset := -\infty$.  In fact, it will be convenient to rephrase \eqref{Eq:l function} slightly differently later, as motivated and discussed in \cref{Sub:Lagr}. This will be  the basis for the non\-linear approach to the d'Alembertian we survey. As we separate these new insights from the current preliminary material, we stick to the traditional definition \eqref{Eq:l function} for the moment.

Physically, given $x,y\in\mms$ the quantity $l(x,y)$ measures the maximal  amount of proper time elapsing when $x$ travels to $y$ through spacetime. 

For every $x,y,z\in\mms$, it satisfies the \emph{reverse triangle inequality}  
\begin{align}\label{Eq:reverse l}
l(x,z) \geq l(x,y) + l(y,z).
\end{align}

For $o\in\mms$, we write the functions $l(o,\cdot)$ and  $l(\cdot,o)$ as $l_o$ and $l^o$, respectively.

\begin{theorem}[Consequences of global hyperbolicity] Assume the spacetime $(\mms,g)$ is globally hyperbolic. Then the following statements hold.
\begin{enumerate}[label=\textnormal{\textcolor{black}{(}\roman*\textcolor{black}{)}}]
\item \textnormal{\textbf{Semicontinuity.}} The function $l$ does not assume the value $\infty$, vanishes on the diagonal of $\mms^2$, and is upper semi\-continuous. Moreover, its positive part $\smash{l_+}$ is con\-ti\-nuous.
\item \textnormal{\textbf{Topology.}} The topology of $\mms$ is determined by $l$. More precisely, the classes $\{I(x,y) : x,y\in\mms\}$ and $\{I^\pm(x) : x\in\mms\}$ are subbases thereof.
\end{enumerate}
\end{theorem}

Despite the duality of \eqref{Eq:l function} to Riemannian distance functions, $l$ is  not a metric. Besides its ``triangle inequality'' being reversed, it is not  definite: for instance, the function $l_o$ vanishes on an entire conical hypersurface $\smash{J^+(o)\setminus I^+(o)}$, where $o\in\mms$. Moreover, $l$ is not  symmetric provided the spacetime $(\mms,g)$ is causal (unless $\mms$ is a singleton, which is excluded by our dimensional hypothesis). On the other hand, it still allows for concepts parallel to their counterparts from Riemannian geometry with a ``metric'' character, notably \emph{geodesics}\footnote{In classical spacetime geometry, this term means a solution to the geodesic equation, hence is more general than the way we use it. Since extremizers of a certain length functional are called geodesics in nonsmooth geometries,  we adopt this terminology from the outset.}. 

\begin{definition}[Geodesic]\label{Def:lgeo} Every timelike maximizer of \eqref{Eq:l function} which is affinely parametrized will  be called an \emph{$l$-geodesic}; here, we will term a causal curve $\gamma$ \emph{affinely parametrized} if  $l(\gamma_s,\gamma_t) = (t-s)\,l(\gamma_0,\gamma_1)$ for every $s,t\in[0,1]$ with $s\leq t$.
\end{definition}

We occasionally use the term of a \emph{proper time parametrized timelike maximizer} $\gamma$ (defined on a real  interval $I$  generally different from $[0,1]$), which instead obeys $l(\gamma_s,\gamma_t) = t-s$ whenever $s,t\in I$ obey $s\leq t$. Every $l$-geodesic $\gamma$ has  a reparametri\-zation by proper time defined on the interval $[0,l(\gamma_0,\gamma_1)]$ (and vice versa).

By basic ODE theory, an $l$-geodesic is smooth. More care is required for maximizers of \eqref{Eq:l function} at whose endpoints $l$ vanishes. Since only timelike maximizers will play a role in this survey, we stick to this simpler case throughout.

\begin{theorem}[Avez--Seifert theorem \cite{avez1963,seifert1967}]\label{Th:Avez-Seifert} If $(\mms,g)$ is globally hyperbolic, any two points $x,y\in\mms$ with $x\ll y$ are connected by an $l$-geodesic.
\end{theorem}

\subsection{Lorentz distance function}\label{Sub:Lorentz} The distance function from an arbitrary subset $\Sigma$ of a metric space is defined as the minimal distance a point has to $\Sigma$.

In spacetime geometry,  even for $\ll$-related points the minimal distance with respect to $l$ is usually $-\infty$ (by travelling along a zigzag curve with some spacelike segment). Instead, comparable to \eqref{Eq:l function} the correct notion of ``distance function'' is obtained by maximizing the length to an appropriate subset $\Sigma$ of $\mms$ with respect to $l$. ``Appropriate'' amounts to three hypotheses on $\Sigma$ described now.

\subsubsection{Achronality} A subset $\Sigma$ of $\mms$ is termed \emph{achronal} if no two points in it are related with respect to $\ll$. This is  equivalent to  $\smash{I^+(\Sigma) \cap I^-(\Sigma) = \emptyset}$.

\begin{definition}[Lorentz distance function] Let $\Sigma$ be an achronal subset of $\mms$. The associated  \emph{Lorentz distance function} $l_\Sigma \colon \mms \to \R\cup\{-\infty,\infty\}$ is 
\begin{align*}
l_\Sigma(x) := \sup_{x^-\in \Sigma} l_+(x^-,x) - \sup_{x^+\in\Sigma} l_+(x,x^+).
\end{align*}
\end{definition}

Here, the hypothesized achronality of $\Sigma$ ensures well-definedness of the above quantity: it implies at most one of the two suprema defining $l_\Sigma(x)$ is  nonzero at each given $x\in\mms$. The function $l_\Sigma$ is positive exactly on $I^+(\Sigma)$, negative exactly on $I^-(\Sigma)$, and identically zero otherwise. A reason for the negative sign on $\smash{I^-(\Sigma)}$ is  in that way, $l_\Sigma$ is monotone along the causal relation $\leq$: if $x,y\in\mms$ obey $x\leq y$, then $\smash{l_\Sigma(x) \leq l_\Sigma(y)}$. (We refer to  \cref{Cor:Steepness} below for a stronger statement.) For simplicity, from now on we focus on $l_\Sigma$ on $\smash{I^+(\Sigma)\cup \Sigma}$, where it is nonnegative (and tailor all further hypotheses on $\Sigma$ relative to this setup, e.g.~future timelike completeness below). Yet, all subsequent discussions generalize   when taking the negative part of $l_\Sigma$ into account; the interested reader is referred to Treude \cite{treude2011}, Treude--Grant \cite{treude-grant2013}, and Cavalletti--Mondino \cite{cavalletti-mondino2020,cavalletti-mondino2022-review} for the background and to Beran et al.~\cite{beran-braun-calisti-gigli-mccann-ohanyan-rott-samann+-} and Braun \cite{braun2024+} for the new results surveyed in this article.

\subsubsection{Future timelike completeness} Given  $\smash{x\in I^+(\Sigma)}$, a point $o\in\Sigma$ is called a \emph{footpoint} of $x$ on $\Sigma$ provided $\smash{l_\Sigma(x) = l(o,x)}$. Not every such $x$ has a footpoint  in general; equivalently, the unique  nonzero supremum in the definition of $l_\Sigma(x)$ may  not be attained. The existence of such footpoints, however, will later ensure the entire set $\smash{I^+(\Sigma)}$ is foliated by negative gradient flow trajectories of $l_\Sigma$  up to a set of measure zero. This explains the need for a sufficient criterion which implies  the existence of footpoints.

We  call a subset $\Sigma$ of $\mms$ \emph{future timelike complete} if for every $\smash{y\in I^+(\Sigma)}$, the closure  of $J^-(y)\cap\Sigma$ relative to $\Sigma$ is compact. This  notion was introduced by Galloway \cite{galloway1986}. It  is weaker than \emph{future causal completeness} of $\Sigma$ subsequently  studied by Treude--Grant \cite{treude-grant2013}, which amounts to replace the  occurrence of ``$I$'' by ``$J$'' in the first clause above. Clearly, if $\Sigma$ is compact, it is future timelike complete.

\begin{lemma}[Footpoint projection]\label{Le:Footpoint proj} Assume that the spacetime $(\mms,g)$ is globally hyperbolic. Let $\Sigma$ be an achronal subset of $\mms$. If $\Sigma$ is future timelike complete, every $\smash{x\in I^+(\Sigma)}$ has a footpoint on  $\Sigma$, i.e.~a point $o\in \Sigma$ such that $\smash{l_\Sigma(x) = l(o,x)}$.
\end{lemma}

\begin{corollary}[Basic regularity]\label{Cor:Steepness} Suppose  $(\mms,g)$ is globally hyperbolic. Let $\Sigma$ be an achronal, future timelike complete subset of $\mms$. Then  the function $\smash{l_\Sigma}$ is real-valued, continuous, and \emph{1-steep} on $\smash{I^+(\Sigma)\cup\Sigma}$, i.e.~for every $\smash{x,y\in I^+(\Sigma)\cup\Sigma}$,
\begin{align}\label{Eq:Satu}
l_\Sigma(y) \geq  l_\Sigma(x) + l(x,y).
\end{align}
\end{corollary}

\subsubsection{Spacelikeness} Traditionally, the geometrically most relevant reference sets for Lorentz distance functions in the literature are achronal  submanifolds of $\mms$. In this part, we  focus on the simpler case when $\Sigma$ is a hypersurface, referring to  Treude \cite{treude2011} and Treude--Grant \cite{treude-grant2013} for general submanifolds. Restricting ourselves to  codimension one  costs no large  generality in terms of applications. For instance, the existence of certain hypersurfaces corresponds to a ``trapped\-ness'' condition in the singularity theorem of Hawking \cite{hawking1967}; see also \cref{Sub:Volume sing} below. 

Thus, let $\Sigma$ be an achronal hypersurface in $\mms$. We  call $\Sigma$  \emph{spacelike} if its normal vector field $\smash{\sfn_{\Sigma}}$ is timelike,  where we tacitly orient $\Sigma$ such that $\smash{\sfn_{\Sigma}}$ is future-directed. The reason for spacelikeness is it will ensure negative gradient flow trajectories  of $l_\Sigma$ emanating from $\Sigma$ (relevant for the disintegration \cref{Th:Disint}) are timelike.

\subsubsection{Future timelike cut locus and smoothness} Are Lorentz distance functions  more regular than \cref{Cor:Steepness} asserts? In Riemannian geometry, it is well-known that distance functions are actually smooth on an open set complementing the cut locus of the reference set in question. 

An analog of this fact holds in Lorentzian geometry. Proofs can  be found in  Treude \cite{treude2011}; a more concise alternative overview is given by Treude--Grant \cite{treude-grant2013}. Let $\Sigma$ be an achronal, future timelike complete, spacelike hypersurface in $\mms$. The \emph{future timelike cut locus} $\smash{\TCut^+(\Sigma)}$ of $\Sigma$ is the closure of the  set of all $\smash{x\in I^+(\Sigma)}$ which are connected  to $\Sigma$ by at least two  $l$-geodesics. Any element of $\smash{\TCut^+(\Sigma)}$ will be called a \emph{future timelike cut point} of $\Sigma$. (In particular, every future timelike cut point has more than one footpoint on $\Sigma$.) The above process of taking the closure  defining $\smash{\TCut^+(\Sigma)}$ adds precisely the set of focal points from $\Sigma$ to the outset.

The future timelike cut locus of $\Sigma$ has measure zero.

\begin{theorem}[Enhanced regularity   {\cite{treude2011,treude-grant2013}}]\label{Th:Enhanced} Let $(\mms,g)$ designate a globally hyperbolic spacetime. Let $\Sigma$ be an achronal, future timelike complete, spacelike hypersurface in $\mms$. Then the following statements hold.
\begin{enumerate}[label=\textnormal{(\roman*)}]
\item \textnormal{\textbf{Smoothness.}} The  function $\smash{l_\Sigma}$ is smooth on $\smash{I^+(\Sigma)\setminus \TCut^+(\Sigma)}$.
\item \textnormal{\textbf{Negative gradient flow.}} Given any $\smash{x\in I^+(\Sigma)\setminus \TCut^+(\Sigma)}$, let $\smash{\gamma}$ form the unique proper-time parametrized timelike maximizer which connects $\Sigma$ to $x$. Then we have $\smash{-(\nabla l_\Sigma)_x = \dot{\gamma}_{l_\Sigma(x)}}$. In particular, $\gamma$ coincides with the negative gradient flow of $l_\Sigma$ passing through $x$ on  $\smash{[0,l_\Sigma(x)]}$.
\item \textnormal{\textbf{Normalization.}} On $\smash{I^+(\Sigma)\setminus \TCut^+(\Sigma)}$, the vector field $\nabla l_\Sigma$ is past-directed and satisfies $\smash{\vert \nabla l_\Sigma \vert^2 = 1}$; in particular, it is timelike. Furthermore, $-\nabla l_\Sigma$ extends to a smooth normal vector field for $\Sigma$.
\end{enumerate}
\end{theorem}

In particular, along each negative gradient flow line of $l_\Sigma$, the inequality \eqref{Eq:Satu} becomes a genuine equality.

Albeit smoothness of Lorentz distance functions fails across the future time\-like cut locus of the reference set in question, they are still locally Lipschitz continuous. This follows from the normalization property from \cref{Th:Enhanced} together  with a well-known locally uniform control (see e.g.~Chru\'sciel \cite{chrusciel2011}*{Lem.~2.6.6}) of the norm of timelike vectors $v\in T\mms$ with respect to the given Riemannian metric $r$  from \cref{Re:Riemannian} as soon as the magnitude $\vert v\vert$ with respect to $g$ [sic] is bounded away from zero. We refer to Sormani--Vega \cite{sormani-vega2016}*{§4} for more details.

\begin{proposition}[Local Lipschitz continuity]\label{Pr:LocalLipschitz} Suppose $(\mms,g)$ is a globally hyperbolic spacetime. Let $\Sigma$ be an achronal, future timelike complete, spacelike hypersurface in $\mms$. Then the function $l_\Sigma$ is locally Lipschitz continuous on $\smash{I^+(\Sigma)}$.
\end{proposition}

\section{Distributional $p$-d'Alembertian}

Throughout the rest of this survey, $p$ and $q$ will be nonzero numbers less than one which are mutually conjugate.  Also, from now on we stipulate global hyperbolicity of $(\mms,g)$ unless explicitly stated otherwise.

We will now motivate and survey the distributional $p$-d'Alembertian recently introduced by Beran--Braun--Calisti--Gigli--McCann--Ohanyan--Rott--Sämann \cite{beran-braun-calisti-gigli-mccann-ohanyan-rott-samann+-}. The distributional view becomes natural in the geometrically relevant case of Lorentz distance functions: at future timelike cut points, they are not sufficiently smooth to have a pointwise second derivative. In the spacetime setting, the ``classically defined'' $p$-d'Alembertian appeared first in  work of  Mondino--Suhr \cite{mondino-suhr2022}. 

\subsection{Radon functionals} We first clarify what we mean by ``distributional''. To this aim, we first fix some terminology from measure theory, referring to classical books   \cite{halmos1950,rudin1966,bogachev2007}  for  details.

By default, throughout this survey a measure is always nonnegative. 

Let $U$ form an open subset of $\mms$. 
The symbol $\Lip_\comp(U)$ denotes  the class of all Lipschitz continuous functions with compact support in $U$. Moreover, $\Vert\cdot\Vert_\infty$ stands for the usual supremum norm.

\begin{definition}[Radon functional] Let $U$ form  an open subset of $\mms$. A \emph{Radon functional} over the subset  $U$ is a linear map $T\colon \Lip_\comp(U)\to \R$ which is continuous in the locally convex topology of uniform convergence on compact subsets of $U$. In other words, the  latter means  for every compact subset $W$ of $U$, there exists a constant $C$ such that for every $\varphi\in\Lip_\comp(U)$ with support  contained in $W$,
\begin{align*}
\big\vert T(\varphi)\big\vert \leq C\,\Vert \varphi\Vert_\infty.
\end{align*}
\end{definition}

Such a Radon functional $T$ is called \emph{nonnegative} provided for every nonnegative function $\smash{\varphi\in \Lip_\comp(U)}$, the value $T(\varphi)$ is nonnegative. 

Evidently, integration with respect to a signed Radon measure defines a Radon functional. Moreover, by the Riesz--Markov--Kakutani representation theorem, every nonnegative Radon functional is given by integration with respect to a uniquely determined Radon measure. Although the difference $\mu - \nu$ of two Radon measures $\mu$ and $\nu$ on $U$ may not be a signed measure (since both $\mu$ and $\nu$ may be infinite), it does make sense as a Radon functional.

\subsection{Lagrangians and Legendre transform}\label{Sub:Lagr} The canonical distance function $\met_r$ on a Riemannian manifold $(\sfN,r)$ is defined by
\begin{align}\label{Eq:Riemanniandistance}
\met_r(x,y) := \inf\!\Big\lbrace\Big[\!\int_0^1 \big\vert\dot\gamma_t\big\vert_r^q\d t\Big]^{1/q} : \gamma\textnormal{  smooth curve with }\gamma_0=x \textnormal{ and }\gamma_1=y \Big\rbrace.
\end{align}
Here $q$ is chosen in $[1,\infty)$. By the Hopf--Rinow theorem, any two points $x,y\in\mms$ are connected by a minimizer if the metric space $(\sfN,\met_r)$ is complete. Moreover,  $\met_r$ does not depend on $q$ in this case. As well-known, the benefit of choosing $q$ larger than one is that minimizers come with a default parametrization by arclength. This fact is connected  to the  \emph{strictly} convex dependence of the Lagrangian $\vert\cdot\vert_r^q/q$ on $\vert\cdot\vert_r$ in this range of $q$. More broadly,  as well-known in calculus of variations, strictly convex minimization problems exhibit unique minimizers (if they exist) for free.

Spacetime geometry deals with maximization problems. This amounts to the requirement of \emph{concavity} properties of the appearing Lagrangians.

Evidently, the Lagrangian in the definition \eqref{Eq:l function} of the time separation function $l$ does not depend  strictly concavely on $\smash{\vert\cdot\vert}$. And indeed, its timelike maximizers do not favor a parametrization per se, by arclength or otherwise. Is there an analog of the preceding Riemannian observations in Lorentzian signature? There is, but  it has entered the stage only surprisingly recently with the contribution of McCann \cite{mccann2020} and sequella by Minguzzi \cite{minguzzi2019-causality} (including negative exponents) and Mondino--Suhr \cite{mondino-suhr2022}. While \cite{minguzzi2019-causality}, among other things, develops a general theory of Legendre transforms induced by abstract Lagrangians on closed cone structures, \cite{mccann2020,mondino-suhr2022} were concerned with the optimal transport problem on spacetimes  as described  below. (The traditional Lagrangian from \eqref{Eq:l function} also appeared in Suhr \cite{suhr2018-theory}.)

Let $q$ be a nonzero number less than one. Define the Lagrangian $\sfL_q$ on $T\mms$ by
\begin{align*}
\sfL_q(v) := \begin{cases} \displaystyle\frac{\vert v\vert^q}{q} & \textnormal{if }v \textnormal{ is future-directed timelike},\\
\displaystyle \lim_{\delta\to 0+} \frac{\delta^q}{q} & \textnormal{if }v\textnormal{ is future-directed lightlike},\\
-\infty & \textnormal{otherwise}.
\end{cases}
\end{align*}
We also consider the Hamiltonian $\sfH_p$ on $T^*\mms$ with
\begin{align*}
\sfH_p(\zeta) := \begin{cases} \displaystyle\frac{\big\vert\zeta\big\vert_*^p}{p} &\textnormal{if }\zeta \textnormal{ is past-directed timelike},\\
\displaystyle\lim_{\delta\to 0+}\frac{\delta^p}{p} & \textnormal{if }\zeta\textnormal{ is past-directed lightlike},\\
-\infty & \textnormal{otherwise}
\end{cases}
\end{align*}
and the corresponding Legendre transform $\smash{\mathcal{L}_p \colon T^*\mms\to T\mms}$ with
\begin{align}\label{Eq:Legendre}
\mathcal{L}_p(\zeta) := \begin{cases} \big\vert \zeta\big\vert_*^{p-2}\,\zeta^\sharp & \textnormal{if }\zeta\textnormal{ is past-directed timelike},\\
0 & \textnormal{otherwise},
\end{cases}
\end{align}
where $p$ is the conjugate exponent to $q$. Note that the limit in the definition of $\sfL_q$ is either $0$ or $-\infty$ depending on the sign of $q$; analogously for $\sfH_p$ and $p$. The precise value of $\smash{\mathcal{L}_p}$ outside each past cone is chosen arbitrarily and will not matter. Given any $\zeta\in T^*\mms$ in the finiteness domain $\smash{\Dom(\sfH_p)}$, $\smash{\mathcal{L}_p(\zeta)}$ constitutes a tangent vector $\smash{v\in \Dom(\sfL_q)}$  that achieves equality in the subsequent reverse Young inequality:
\begin{align}\label{Eq:Reverseyoung}
\zeta(v) \geq \frac{\big\vert \zeta\big\vert_*^p}{p} + \frac{\vert v\vert^q}{q}.
\end{align}

\begin{proposition}[Strict concavity \cite{mccann2020,minguzzi2019-causality,mondino-suhr2022}] The Lagrangian $\sfL_q$ is concave. On the interior of each future cone, it is strictly concave and smooth; in particular, the Legendre transform $\smash{\mathcal{L}_p(\zeta)}$ of any given $\smash{\zeta\in\Dom(\sfH_p)}$ is the unique element $\smash{v\in\Dom(\sfL_q)}$ achieving equality in \eqref{Eq:Reverseyoung}. Furthermore, the restriction of $\smash{\mathcal{L}_p}$ to  $\smash{\mathcal{L}_p^{-1}(T^*\mms\setminus \{0\})}$ is bijective with smooth inverse. 
\end{proposition}

We now consider the function $\smash{\widehat{l}_q\colon\mms^2 \to \R\cup\{-\infty,\infty\}}$ given by
\begin{align*}
\widehat{l}_q(x,y) := \sup\!\Big\lbrace\!\int_0^1 \sfL_q(\dot\gamma_t)\d t: \gamma\textnormal{ causal curve with }\gamma_0 = x \textnormal{ and }\gamma_1= y\Big\rbrace
\end{align*}
If $(\mms,g)$ is globally hyperbolic, given any $x,y\in\mms$ such that $x\leq y$,  by a version of \cref{Th:Avez-Seifert} there exists a maximizer of the previous supremum. It does not depend on $q$ and, when $x\ll y$, is an $l$-geodesic. Consequently, for every $x,y\in\mms$ the quantity $\smash{[q\,\widehat{l}_q(x,y)]^{1/q}}$ coincides with  $l(x,y)$ from \eqref{Eq:l function}, cf.~e.g.~McCann \cite{mccann2020} --- provided we adopt natural infinity and power conventions outside the relation $\ll$. This derives the following representation of the time separation function $l$ for every nonzero number $q$ no larger than one, where $x,y\in\mms$:
\begin{align*}
l(x,y) = \sup\!\Big\lbrace\Big[\!\int_0^1 \vert\dot\gamma_t\vert^q\d t\Big]^{1/q} : \gamma\textnormal{ causal curve with }\gamma_0=x\textnormal{ and }\gamma_1=y \Big\rbrace.
\end{align*}

\subsection{The operator} Now we come to the main object of our survey, the distributional $p$-d'Alembertian, where $p$ is a nonzero number less than one.

Recall the usual \emph{d'Alembertian} (also called \emph{wave operator} or \emph{box operator}) from general relativity is defined by $\Box := \div \circ\, \nabla$. It depends on the Lorentzian metric $g$ and the reference measure $\vol$. By  \eqref{Eq:div def}, it obeys an   integration by parts formula. 

The above discussion, on the other hand, motivates the distributional $p$-d'Alem\-bertian,  formally the divergence of the $p$-gradient $\smash{\mathcal{L}_p\circ\rmd}$. The subsequent notion was  proposed by Beran--Braun--Calisti--Gigli--McCann--Ohanyan--Rott--Sämann \cite{beran-braun-calisti-gigli-mccann-ohanyan-rott-samann+-} inspired by Gigli's distributional Laplacian on metric measure spaces \cite{gigli2015} \MB{and its $p$-Laplacian counterpart (in the classical range of $p$) by Gigli--Mondino \cite{gigli-mondino2013}}.

\begin{definition}[Distributional $p$-d'Alembertian \cite{beran-braun-calisti-gigli-mccann-ohanyan-rott-samann+-}]\label{Def:Distrdalem}  Let $\u$ be a locally Lipschitz continuous function which is $l$-causal, i.e.~nondecreasing along the causal relation $\leq$. We say $\u$ lies in the domain of the $p$-d'Alembertian on $U$, symbolically $\u\in \Dom(\BOX_p\mres U)$, if there exists a Radon functional $T$ over $U$ such that for every $\varphi\in \Lip_\comp(U)$,
\begin{align}\label{Eq:IBPinequ}
\int_\mms \rmd \varphi(\nabla\u)\,\big\vert\rmd\u\big\vert_*^{p-2}\d\vol \leq -T(\varphi). 
\end{align}
\end{definition}

The distributional $p$-d'Alembertian induced by a smooth measure $\meas$ on $\mms$ other than $\vol$ is defined analogously.

For any $T$ as above, we will write $\smash{T\in\BOX_p\u\mres U}$.

\begin{remark}[Relation to the classical d'Alembertian]\label{Re:Indep!} 
Of course, if $\u$ is a Lorentz distance function $l_\Sigma$ as in \cref{Sub:Lorentz}, by \cref{Th:Enhanced} we have $\smash{\vert\rmd \u\vert_*=1}$ on $\smash{I^+(\Sigma)}$ outside a set of measure zero. In this case, its distributional $p$-d'Alembertian does not depend on $p$. In particular, it describes the usual d'Alembertian distributionally, which makes the $p$-version retain the geometric relation of the d'Alembertian of $l_\Sigma$ to the forward mean curvature of $\Sigma$ (if $\Sigma$ is a suitable  hypersurface). \hfill$\blacksquare$
\end{remark}

\begin{remark}[Infinitesimally strict concavity  \cite{braun2024+}]\label{Re:Uniquenessdalem} Why does \cref{Def:Distrdalem} require the ``integration by parts formula'' to hold with an inequality? In our spacetime framework, one easily realizes \eqref{Eq:IBPinequ} must be an identity (by substituting the test function $\varphi$ by $-\varphi$). This property holds more generally on so-called infinitesimally Minkowskian metric measure spacetimes (cf.~\cref{Def:InfMink} below) by Beran--Braun--Calisti--Gigli--McCann--Ohanyan--Sämann \cite{beran-braun-calisti-gigli-mccann-ohanyan-rott-samann+-}. It is also true on Finsler spacetimes or so-called infinitesimally strictly concave metric measure spacetimes, a terminology introduced by Braun \cite{braun2024+} inspired by the infinitesimally strictly convex metric measure spaces of Gigli  \cite{gigli2015}. 

In general, however, the distributional $p$-d'Alembertian may well  be multivalued, and the inequality \eqref{Eq:IBPinequ} compensates for that. As already realized in Riemannian  signature by Gigli \cite{gigli2015}, the chosen names from \cite{gigli2015,braun2024+} for infinitesimally strict  concavity (or convexity) ruling out this ambiguity explain what may go wrong. Roughly speaking, provided the Lorentzian ``norm'' on each tangent space (such as $\vert\cdot\vert$) fails to be strictly concave, the Legendre transform may become multivalued. Consequently, the differential of the  function $\u$ in question may have more than one $p$-gradient, and there is no canonical way to decide which of these to insert into the divergence in the integration by parts identity  \eqref{Eq:div def}. 

Examples of metric measure spacetimes which are not infinitesimally strictly concave are given in Beran--Braun--Calisti--Gigli--McCann--Ohanyan--Rott--Sämann \cite{beran-braun-calisti-gigli-mccann-ohanyan-rott-samann+-} and Gigli \cite{gigli+}. Incidentally, some of these still satisfy synthetic timelike Ricci curvature bounds from \cref{Def:TCD} below.\hfill$\blacksquare$
\end{remark}

Braun \cite{braun2024+} shows several basic properties of the distributional $p$-d'Alembertian. For instance, these include
\begin{itemize}
\item its $(p-1)$-homogeneity,
\item a restriction property, 
\item a chain rule,
\item a transformation formula by change of the reference measure, and
\item a nonsmooth notion of $p$-harmonicity.
\end{itemize}

\section{Abstract existence and applications}\label{Sub:Abstractex}

Now we review the recent d'Alembert comparison theorem in a weak formulation by Beran--Braun--Calisti--Gigli--McCann--Ohanyan--Rott--Sämann \cite{beran-braun-calisti-gigli-mccann-ohanyan-rott-samann+-}. We survey the main underlying ideas and describe how \cite{beran-braun-calisti-gigli-mccann-ohanyan-rott-samann+-} derives the general existence of the distributional $p$-d'Alembertian of Lorentz distance functions  to a point,  appropriate powers thereof, and more general Kantorovich potentials. 

\subsection{Prerequisite. Lorentz--Wasserstein distance and Kantorovich duality}\label{Sub:Wasserstein} A key set of techniques comes from the theory of optimal transport on space\-times. The central question we make precise below is: given two mass distributions on a spacetime $(\mms,g)$ --- one in the ``chronological future'' of the other --- what is the best way to move the past into the future portion, subject to a given cost function? (The similarities to the maximization problem \eqref{Eq:l function} are not coincidental.) This optimization problem defines a ``time separation function'' on the space of Borel probability measures $\Prob(\mms)$ comparable to the famous Wasserstein distance. More broadly, the optimal transport problem in numerous geometries has a long history that even has its roots before the French revolution; we refer the interested reader to the book of Villani \cite{villani2009} for more background and applications.

The optimal transport problem in special relativity was studied by Brenier \cite{brenier2003}, Bertrand--Puel \cite{bertrand-puel2013}, and Bertrand--Pratelli--Puel \cite{bertrand-pratelli-puel2018}. In general relativity, first works were executed by  Suhr \cite{suhr2018-theory} and Kell--Suhr \cite{kell-suhr2020}. There are fascinating links between  optimal transport and  the  early universe reconstruction problem, cf.~the introduction of Cavalletti--Mondino \cite{cavalletti-mondino2022-review}. Eckstein--Miller \cite{eckstein-miller2017}  pioneered the analog of the Wasserstein distance on spacetimes, the so-called  \emph{$q$-Lorentz--Wasserstein distance}. Moreover, they established  causality properties of the ``spacetime'' $\Prob(\mms)$ endowed with this $q$-Lorentz--Wasserstein distance, extensions of which were given by Suhr \cite{suhr2018-theory}, Braun--McCann \cite{braun-mccann2023},  and Beran--Braun--Calisti--Gigli--McCann--Ohanyan--Rott--Sämann \cite{beran-braun-calisti-gigli-mccann-ohanyan-rott-samann+-}. Lastly, McCann \cite{mccann2020}, Mondino--Suhr \cite{mondino-suhr2022}, and Cavalletti--Mondino \cite{cavalletti-mondino2020} understood uniqueness of optimal transport plans, Kantorovich duality, and the connection of ``chronological optimal transport'' to timelike Ricci curvature in smooth \cite{mccann2020,mondino-suhr2022} and nonsmooth \cite{cavalletti-mondino2020} settings. We elaborate on these aspects in the sequel in more depth. For more details or heuristics about Lorentzian optimal transport, we refer to the survey of Cavalletti--Mondino \cite{cavalletti-mondino2022-review}.

\subsubsection{Lorentz--Wasserstein distance and geodesics} Given a Polish space $E$, $\Prob(E)$ denotes the class of all Borel probability measures on $E$. It is endowed with the narrow topology, cf.~e.g.~Billingsley \cite{billingsley1999} for background.

Let $F$ be another Polish space. For $\mu\in\Prob(E)$ and a Borel map $T\colon E\to F$, the \emph{push-forward} $T_\push\mu\in\Prob(F)$ of $\mu$ under $F$ is  given by $T_\push\mu[B] := \mu[T^{-1}(B)]$ for every Borel subset $B$ of $F$. 

For two probability measures $\mu,\nu\in\Prob(\mms)$ on the spacetime $(\mms,g)$, an element $\pi\in\Prob(\mms^2)$ is called a \emph{coupling} of $\mu$ and $\nu$ if $(\pr_1)_\push\pi = \mu$ and $(\pr_2)_\push\pi=\nu$ (where the inherent projection maps have been defined in \cref{Sub:Some}); we term $\mu$ and $\nu$ the first and second marginal of $\pi$, respectively. We call such a coupling \emph{chronological} and \emph{causal} if it is concentrated on the subsets $\{l>0\}$ and $\{l\geq 0\}$ of $\mms^2$, respectively. The product measure $\mu\otimes\nu$ is always a coupling of $\mu$ and $\nu$, yet in general it is not even causal. Eckstein--Miller \cite{eckstein-miller2017} give interesting sufficient conditions for the existence of causal couplings on spacetimes between given marginals in the spirit of martingale optimal transport. By requiring the existence of a causal coupling, one can define a natural causal relation on $\Prob(\mms)$ whose study has been initiated by Eckstein--Miller \cite{eckstein-miller2017} and continued by Suhr \cite{suhr2018-theory}, Braun--McCann \cite{braun-mccann2023}, and Beran--Braun--Calisti--Gigli--McCann--Ohanyan--Rott--Sämann \cite{beran-braun-calisti-gigli-mccann-ohanyan-rott-samann+-}.

A  ``spacetime'' geometry on $\Prob(\mms)$ is given by the  $q$-Lorentz--Wasserstein distance $\ell_q\colon \Prob(\mms^2)\to \R_+\cup\{-\infty,\infty\}$, where \MB{$q\in (-\infty,1)\setminus\{0\}$ and}
\begin{align*}
\ell_q(\mu,\nu) := \sup\!\Big\lbrace\Big[\!\int_{\mms^2} l^q\d\pi\Big]^{1/q} : \pi\textnormal{ causal coupling of }\mu \textnormal{ and }\nu \Big\rbrace;
\end{align*}
as usual, we impose the  standing convention $\sup\emptyset := -\infty$. The reader should note the evident analogy to \eqref{Eq:l function}. By a standard gluing argument, $\ell_q$ obeys  the reverse triangle inequality, cf.~e.g.~Eckstein--Miller \cite{eckstein-miller2017}. 

A coupling $\pi$ of $\mu$ and $\nu$ will be called  \emph{$\ell_q$-optimal} if it is causal and it attains the above supremum. As concerns existence of $\ell_q$-optimal couplings, for our purposes the following consequence of standard calculus of variations  will suffice.

\begin{lemma}[Existence of $\ell_q$-optimal couplings]\label{Le:EX!}  Suppose $\mu,\nu\in\Prob(\mms)$ are compactly supported within a globally hyperbolic spacetime $(\mms,g)$. If $\mu$ and $\nu$ admit a causal coupling,  then they automatically admit an $\smash{\ell_q}$-optimal coupling.
\end{lemma}

As we will anyway deal only with compactly supported probability measures in the applications we have in mind, we restrict our attention to them from now on. We refer to Cavalletti--Mondino \cite{cavalletti-mondino2020} for more general notions and results.

Eventually, we shall be interested in movement through spacetime in timelike directions. So what about transport along the chronological relation $\ll$?  Following Cavalletti--Mondino \cite{cavalletti-mondino2020}, we call a pair $(\mu,\nu)$ of elements $\mu,\nu\in\Prob(\mms)$ with compact support \emph{timelike $q$-dualizable} if $\mu$ and $\nu$ have a chronological $\ell_q$-optimal coupling. Evidently, this implies $\smash{\ell_q(\mu,\nu) > 0}$; however, positivity of $\smash{\ell_q(\mu,\nu)}$ does not imply timelike $q$-dualizability of $(\mu,\nu)$ (as a nonnegligible portion of mass may be transported in lightlike directions). In general, it appears hard to handle chronology of an $\ell_q$-optimal coupling, as chronology is not stable under narrow limits. However, suppose that a pair $(\mu,\nu)$ of compactly supported elements $\mu,\nu\in\Prob(\mms)$ satisfies $\supp\mu\times\supp\nu \subset \{l>0\}$. This amounts to say that \emph{every} point $x\in \supp\mu$ lies in the chronological past of \emph{every} point $y\in\supp\nu$. Then by \cref{Le:EX!}, the pair $(\mu,\nu)$ is timelike $q$-dualizable. This observation will suffice for our purposes. How about uniqueness of chronological $\ell_q$-optimal couplings, provided they exist? If one of the marginals is a Dirac mass, the answer is clearly positive (as there is always exactly one element of $\Prob(\mms^2)$ coupling a Dirac mass with anything else). And indeed, for most of our surveyed content we will consider chronological $\ell_q$-optimal transport from a Dirac mass, cf.~e.g.~\cref{Sub:dAlembert comparison}. The uniqueness question for  more general marginals is much more elaborate; we refer to McCann \cite{mccann2020} and Cavalletti--Mondino \cite{cavalletti-mondino2020} (see also Braun \cite{braun2023-renyi}) for results in smooth and nonsmooth structures. The corresponding \cref{Th:Optimal} below by McCann \cite{mccann2020} is enough for us.

\begin{definition}[Geodesic of mass distributions \cite{mccann2020}]\label{Def:lqgeos} A mapping $\mu\colon[0,1]\to\Prob(\mms)$ will be called \emph{$\ell_q$-geodesic} provided $\ell_q(\mu_0,\mu_1) \in (0,\infty)$ and every $s,t\in[0,1]$ such that $s\leq t$ satisfy $\smash{\ell_q(\mu_s,\mu_t) = (t-s)\,\ell_q(\mu_0,\mu_1)}$.
\end{definition}

The reader should note the analogy to \cref{Def:lgeo}.

Sufficient conditions for the narrow continuity of $\ell_q$-geodesics are given in Braun--McCann \cite{braun-mccann2023} and Beran--Braun--Calisti--Gigli--McCann--Ohanyan--Rott--Sämann \cite{beran-braun-calisti-gigli-mccann-ohanyan-rott-samann+-}, based on a similar observation of McCann \cite{mccann2023-null} about the continuity of $l$-geodesics in abstract metric measure spacetimes.

It is frequently convenient to ``represent a geodesic of probability measures as a probability measure on geodesics''. We call an element $\bdpi$ on $\Prob(C([0,1];\mms))$ an \emph{$\ell_q$-optimal dynamical coupling} of two  measures $\mu,\nu\in\Prob(\mms)$ if it is concentrated on the set of $l$-geodesics and the joint endpoint projection $(\eval_0,\eval_1)_\push\bdpi$ constitutes an  $\ell_q$-optimal coupling of $\mu$ and $\nu$; here, given any $t\in[0,1]$ the evaluation map $\eval_t\colon C([0,1];\mms)\to \mms$ is defined through $\eval_t(\gamma) := \gamma_t$. For mild sufficient conditions guaranteeing such a representation --- often called ``lifting theorem'' --- we refer to Braun--McCann \cite{braun-mccann2023} and Beran--Braun--Calisti--Gigli--McCann--Ohanyan--Rott--Sämann \cite{beran-braun-calisti-gigli-mccann-ohanyan-rott-samann+-}. (The latter reference deals with liftings of more curves of probability measures which are more general than $\ell_q$-geodesics.) In the relevant case for our survey, it can be constructed very naturally, cf.~the proof of \cref{Th:DAlembertcomparison}.

\subsubsection{Kantorovich duality} The classical optimal transport problem is well-known to have a dual formulation as an infinite-dimensional linear program: the so-called \emph{Kantorovich duality} set up by the Soviet mathematician Kantorovich. For a general account, we refer again to the book of Villani \cite{villani2019}. On spacetimes, Kantorovich duality was first studied for the more classically motivated yet not ``strictly concave'' $1$-Lorentz--Wasserstein distance by Suhr \cite{suhr2018-theory} and Kell--Suhr \cite{kell-suhr2020}. A nonsmooth analog of the famous Kantorovich--Rubinstein duality formula for $\ell_1$  was derived by Braun--McCann \cite{braun-mccann2023}. In the strictly concave situation relevant for our survey, Lorentzian Kantorovich duality has been studied thoroughly by McCann \cite{mccann2020}, Mondino--Suhr \cite{mondino-suhr2022}, and Cavalletti--Mondino \cite{cavalletti-mondino2020}. The results  of \cite{cavalletti-mondino2020} also generalize to abstract metric measure spacetimes. Braun--Ohta \cite{braun-ohta2024} have derived properties analogous to those of \cite{mccann2020} for Finsler spacetimes. Further heuristics are outlined in the survey \cite{cavalletti-mondino2022-review} of Cavalletti--Mondino.

In our case, given $\mu,\nu\in \Prob(\mms)$ the  Kantorovich duality  formula reads
\begin{align}\label{Eq:Kantor}
\begin{split}
\frac{\ell_q(\mu,\nu)^q}{q} &= \inf\!\Big\lbrace\!\int_\mms \u\d\mu +\int_\mms\v\d\nu :\\
&\qquad\qquad \u\in L^1(\mms,\mu),\,\v\in L^1(\mms,\nu), \frac{l^q}{q}\leq \u\oplus\v \textnormal{ on }\mms^2\Big\rbrace;
\end{split}
\end{align}
here, the function $\u\oplus\v$ on $\mms^2$ is defined by $\u\oplus\v(x,y) := \u(x) + \v(y)$. For nice heuristics about this formula in more traditional optimal transport, we refer the reader to Villani \cite{villani2009}*{§5} for an economic picture and Ambrosio--Gigli \cite{ambrosio-gigli2011}*{§1.3} for a more mathematical derivation.

Does there exist a minimizer of the right-hand side of \eqref{Eq:Kantor}? As pointed out by McCann \cite{mccann2020}  and Cavalletti--Mondino \cite{cavalletti-mondino2020}, the application of classical Kantorovich duality to this existence result is not straightforward, since the cost function $l^q/q$ degenerates outside the causal relation $\leq$. The works \cite{mccann2020,cavalletti-mondino2020} propose different  sufficient conditions on $\mu$ and $\nu$ which ensure that the infimum in \eqref{Eq:Kantor} is attained. They all hold under the strong yet drastically simpler hypothesis when $\mu$ and $\nu$ are compactly supported and satisfy $\supp\mu\times\supp\nu\subset\{l>0\}$, which is what we assume for the rest of this section to simplify the presentation.

Why bother about this existence question? The punchline is Kantorovich potentials can be used to describe $\ell_q$-optimal couplings and $\ell_q$-geodesics explicitly by flow techniques. Suppose $\u$ and $\v$ attain the infimum \eqref{Eq:Kantor}, and assume $\u$ was smooth. Let $\pi$ be an $\ell_q$-optimal coupling of $\mu$ and $\nu$ (which is necessarily chronological). Using the identity \eqref{Eq:Kantor} and the fact that $\pi$ couples $\mu$ and $\nu$,
\begin{align*}
\int_{\mms^2} \frac{l^q}{q}\d\pi &=\frac{\ell_q(\mu,\nu)^q}{q} = \int_\mms\u\d\mu + \int_\mms\v\d\nu= \int_{\mms^2} \big[\u\circ\pr_1 + \v\circ\pr_2\big]\d\pi.
\end{align*}
Since $\smash{l^q/q\leq \u\oplus\v}$ on $\mms^2$, this forces
\begin{align}\label{Eq:Ulk}
\u\circ\pr_1 + \v\circ\pr_2= \frac{l^q}{q} \quad\pi\textnormal{-a.e.}
\end{align}
Hence, optimality of $\pi$ forces a certain $\pi$-a.e.~rigidity of the inequality defining $\u$ and $\v$. This  has deep structural consequences which, roughly speaking, arise as follows. Let $\exp^r$ denote the exponential map induced by the Riemannian metric $r$ from \cref{Re:Riemannian}. Let $x\in\supp\mu$, $y\in\supp\nu$, and $\xi\in T_x\mms$ have unit norm $\vert\xi\vert_r=1$. Given $\delta > 0$ sufficiently close to zero, the defining property of $\u$ and $\v$ combined with \eqref{Eq:Ulk} and a Taylor expansion of $\u$ around $x$ yield
\begin{align*}
\frac{l(\exp_x^r(\delta\,\xi),y)^q}{q} &\leq \u(\exp_x^r(\delta\,\xi)) + \v(y)\\
&= \u(x) + \delta\,\rmd\u(\xi) + \v(y) + \rmo(\delta)\\
&= \frac{l(x,y)^q}{q} + \delta\,\rmd\u(\xi) + \rmo(\delta).
\end{align*}
Hence, the function $(l^y)^q/q$ is superdifferentiable at $x$ with supergradient $\rmd\u$. On the other hand, general  smoothness properties of $l$ recalled in the proof sketch of \cref{Th:Optimal} below yield subdifferentiability at $x$ too. By standard convex analysis, this means $(l^y)^q/q$ is differentiable at $x$. Moreover,  all super- and subdifferentials agree, which gives $\rmd (l^y)^q/q = \rmd\u$ at $x$ by differentiation. On the other hand, if $y$ is not in the future timelike cut locus of $x$, by Legendre--Fenchel duality $\rmd(l^y)^q/q$ coincides with $\smash{\mathcal{L}_q(-\log_x y)}$. Inverting this equation  forces $y = \exp_x(-t\,\mathcal{L}_p(\rmd\u))$ for $\pi$-a.e.~$(x,y)\in\mms^2$. This means $\pi$ is in fact concentrated on the graph of the geodesic flow map $\Psi_1$ of the negative $p$-gradient of $\u$ flown up to time one! 

On Euclidean space, this line of thought originates in the characterization of optimal couplings for the quadratic cost by Brenier \cite{brenier1989}. His result was generalized to general Riemannian manifolds by McCann \cite{mccann2001}; analogous formulas for more general cost functions are obtained by Gangbo--McCann \cite{gangbo-mccann1996}.

Following Cavalletti--Mondino \cite{cavalletti-mondino2020}, given any two subsets $U$ and $V$ of $\mms$, a Borel function $\u$ on $U$ will be called \emph{$l^q/q$-concave} relatively to $(U,V)$ if there exists a function $\v$ on $V$ such that for every $x\in U$,
\begin{align*}
\u(x) =\inf_{y\in V} \Big[\frac{l(x,y)^q}{q}-\v(y)\Big].
\end{align*}
We term the function $\smash{\varphi^{(l^q/q)}}$ on $V$ given by
\begin{align*}
\u^{(l^q/q)}(y) := \sup_{x\in U} \Big[\frac{l^q(x,y)}{q}-\u(x)\Big]
\end{align*}
the \emph{$\smash{l^q/q}$-transform} of $\u$ relative to $(U,V)$. 

\begin{definition}[Strong Kantorovich duality  \cite{cavalletti-mondino2020}]\label{Def:Strong Kantorovich}  A pair $(\mu,\nu)$ of elements $\mu,\nu\in\Prob(\mms)$ with compact support are said to satisfy \emph{strong $l^q/q$-Kantorovich duality} if $\ell_q(\mu,\nu)>0$ and there is a Borel function $\u$ on $\supp\mu$ which is $l^q/q$-concave relatively to $(\supp\mu,\supp\nu)$ such that
\begin{align*}
\frac{\ell_q(\mu,\nu)^q}{q} = \int_\mms \u^{(l^q/q)}\d\nu + \int_\mms \u\d\mu;
\end{align*}
in this case, $\u$ is called \emph{strong Kantorovich potential} relative to $(\supp\mu,\supp\nu)$.
\end{definition}

The last ingredient for the characterization of optimal couplings and $\ell_q$-geodesics on spacetimes from \cref{Th:Optimal} below is the singular set of the time separation function $l$. A control on it will be crucial to extend the d'Alembert comparison theorem across timelike cut loci as sketched in the proof of  \cref{Th:DAlembertcomparison}.

\begin{definition}[Singular set \cite{mccann2020}] The singular set $\sing\, l$ of $l$ constitutes the set of all pairs $(x,y)\in\mms^2$ such that $x \not\ll y$ or one of the points $x$ or $y$ does not lie in the relative interior of an $l$-geodesic through $\mms$.
\end{definition}

Observe that the projection of $\sing\, l \cap (\{o\} \times I^+(o))$ onto the second coordinate contains the future timelike cut locus of the given point $o\in\mms$.

\begin{theorem}[Uniqueness of optimal couplings and $\ell_q$-geodesics \cite{mccann2020}]\label{Th:Optimal} Let $(\mms,g)$ be a globally hyperbolic spacetime. Assume  the pair $(\mu_0,\mu_1)$ of compactly supported elements $\mu_0,\mu_1\in\Prob(\mms)$  satisfies $\supp\mu_0\times\supp\mu_1\subset\{l>0\}$. Furthermore, suppose $\mu_0$ is $\vol$-absolutely continuous \textnormal{(}hence absolutely continuous with respect to any smooth Borel measure $\meas$ on $\mms$\textnormal{)}. Then the following hold.
\begin{enumerate}[label=\textnormal{(\roman*)}]
\item \textnormal{\textbf{Strong Kantorovich duality.}} The pair $(\mu_0,\mu_1)$ obeys strong $l^q/q$-Kanto\-rovich duality according to \cref{Def:Strong Kantorovich}. In addition, it admits a locally semiconvex and hence locally Lipschitz continuous strong Kantorovich potential $\u$ relative to $(\supp\mu_0,\supp\mu_1)$.
\item \textnormal{\textbf{Singular set.}} Let $\Psi$ be the  flow of the negative $p$-gradient of $\u$, i.e.~
\begin{align*}
\Psi_t(x) := \exp_x(-t\,\mathcal{L}_p(\rmd \u)),
\end{align*}
whenever this expression makes sense, where $\mathcal{L}_p$ is the Legendre transform from \eqref{Eq:Legendre}. Then for $\mu_0$-a.e.~$x\in\mms$ the flow $\Psi_\cdot(x)$ is defined  on the interval $[0,1]$ and we have  $(x,\Psi_1(x))\notin \sing\,l$. 
\item \textnormal{\textbf{Uniqueness of optimal couplings.}} The coupling $\pi := (\Id, \Psi_1)_\push\mu_0$  is the unique --- necessarily chronological --- $\ell_q$-optimal coupling of $\mu_0$ and $\mu_1$; in particular, $\mu_1$ is necessarily obtained by flowing $\mu_0$ along the directive $\Psi$ up to time one, viz.~$\smash{\mu_1 = (\Psi_1)_\push\mu_0}$.
\item \textnormal{\textbf{Uniqueness of geodesics.}} For     $t\in [0,1)$, the expression $\mu_t := (\Psi_t)_\push\mu_0$ is well-defined, $\vol$-absolutely continuous, and compactly supported. Moreover, the curve $\mu\colon[0,1]\to\Prob(\mms)$ thus defined constitutes the unique $\smash{\ell_q}$-geodesic joining $\mu_0$ to $\mu_1$ according to \cref{Def:lqgeos}.
\end{enumerate}
\end{theorem}

The chronology hypothesis on the marginals can be suitably  relaxed. Mostly, working under this quite strong assumption suffices (and is quite convenient), as will be the case in our survey. Note, however, that it does not propagate to the relative interior of $\ell_q$-geodesics, as their supports will usually overlap between ``nearby'' intermediate points of the $\ell_q$-geodesic in question.  

An analog of \cref{Th:Optimal} holds for Finsler spacetimes, cf.~Braun--Ohta \cite{braun-ohta2024}.

\begin{proof}[Proof sketch of \cref{Th:Optimal}] Roughly speaking, the exact shape of chronological $\ell_q$-optimal couplings and $\ell_q$-geodesics claimed in the last two items  follow from the above heuristics. We briefly comment on some regularity aspects which essentially make this argumentation work. This will clarify the role of the singular set $\sing\,l$.

One first demostrates the time separation function $l$ is locally semiconvex\footnote{In \cite{braun-gigli-mccann-ohanyan-samann+} Braun--Gigli--McCann--Ohanyan--Sämann used a variant of this observation to improve previous regularity properties of the spacetime Busemann function due to Galloway--Horta \cite{galloway-horta1996} from local  Lipschitz continuity to local semiconcavity.} on $\ll$ and smooth precisely on $\{l>0\}\setminus \sing\,l$. The local semiconvexity propagates to any strong Kantorovich potential $\varphi$ by the  chronology hypothesis on $\mu_0$ and $\mu_1$ --- under which standard Kantorovich duality yields its actual existence, cf.~e.g.~Villani \cite{villani2009}. This is the first claim.

Local semiconvexity of $\u$ yields its twice differentiability $\vol$-a.e.~by Alexandrov's theorem. Since $\mu_0$ is $\vol$-absolutely continuous, $\u$ is twice differentiable $\mu_0$-a.e.; in particular, $\rmd\varphi$ makes sense $\mu_0$-a.e. Now let $(x,y)\in\supp\pi$, where $\pi$ is a chronological $\ell_q$-optimal coupling given by  \cref{Le:EX!}. If $x$ is  a twice differentiability point of $\u$, one can show $l$ is semiconcave around it. However, semiconcavity of $l$ fails exactly on $\sing\,l$ (somewhat complementing its   smoothness on $\{l>0\}\setminus\sing\,l$). This gives the desired property $(x,y)\notin\sing\,l$. 
\end{proof}

\subsection{Prerequisite. Timelike curvature-dimension condition} Armed with the exact shape of $\ell_q$-geodesics (and the Monge--Ampère equation, cf.~McCann \cite{mccann2020}), one can now proceed to study the influence of Ricci curvature along chronological optimal transport through spacetime. The picture, nicely illustrated by the lazy gas experiment by Villani \cite{villani2009}*{p.~460}, is that e.g.~nonnegative timelike Ricci curvature makes mass --- that we know is transported along $l$-geodesics by \cref{Th:Optimal} --- initially spread out and then refocus on  the way from past to future. This results in \emph{convexity} of entropy, a measure of mass concentration.

Given a smooth measure $\meas$ on $\mms$, define the \emph{Boltzmann entropy} $\Ent_\meas\colon\Prob(\mms)\to \R\cup [-\infty,\infty]$ through
\begin{align*}
\Ent_\meas(\mu) := \begin{cases}\displaystyle\lim_{\varepsilon \to 0+}\int_{\{\rmd\mu/\rmd\meas \geq \varepsilon\}} \frac{\rmd\mu}{\rmd\meas}\log\frac{\rmd\mu}{\rmd\meas}\d\meas & \textnormal{if }\mu\ll\meas,\\
\infty & \textnormal{otherwise}.
\end{cases}
\end{align*}

\MB{The following definition was given by Cavalletti--Mondino \cite{cavalletti-mondino2020} for $q\in (0,1)$, but it adapts seemlessly to cover the case $q\in (-\infty,0)$ as well.}

\begin{definition}[Timelike curvature-dimension condition \cite{cavalletti-mondino2020}]\label{Def:TCD} We say $(\mms,g,\meas)$ satisfies the entropic timelike curvature-dimension condition, briefly $\smash{\TCD_q^e(K,N)}$, if the following property holds. For all timelike $p$-dualizable pairs $(\mu_0,\mu_1)$ of compactly supported, $\meas$-absolutely continuous elements $\mu_0,\mu_1\in\Prob(\mms)$, there exist
\begin{itemize}
\item an $\smash{\ell_q}$-geodesic $\mu \colon [0,1]\to\Prob(\mms)$ from $\mu_0$ to $\mu_1$ and
\item a chronological $\ell_q$-optimal coupling $\pi$ of $\mu_0$ and $\mu_1$
\end{itemize}
such that the assignment $e(t) := \Ent_\meas(\mu_t)$ depends semiconvexly on $t\in (0,1)$ and satisfies the following differential inequality distributionally on $[0,1]$:
\begin{align*}
e''(t) - \frac{e'(t)^2}{N}\geq K\int_{\mms^2}l^2\d\pi.
\end{align*}
\end{definition}

Although the chosen exponent $q$ here  relates to the Lagrangian from \cref{Sub:Lagr}, its precise choice will not matter in any of our results. Anyway, the TCD condition is expected to be  independent of the  exponent $q$ in several   relevant situations, as conjectured in the survey of Cavalletti--Mondino \cite{cavalletti-mondino2022-review}.

The key insight leading to \cref{Def:TCD} was the following \cref{Th:EquivTCD} proven independently by McCann \cite{mccann2020} and Mondino--Suhr \cite{mondino-suhr2022}. This result was foreshadowed by its Riemannian precedent of Cordero-Erausquin--McCann--Schmucken\-schläger \cite{cordero-erausquin-mccann-schmuckenschlager2001} and von Renesse--Sturm \cite{von-renesse-sturm2005} (indicated before by formal computations of Otto--Villani \cite{otto-villani2000}), which lead to the celebrated curvature-dimension conditions for metric measure spaces by Sturm \cite{sturm2006-i,sturm2006-ii} and Lott--Villani \cite{lott-villani2009}. As for these works, the main point of \cref{Def:TCD} is it makes sense on metric measure spacetimes $(\mms,l,\meas)$, where $l$ is an abstract time separation function not necessarily of the form \eqref{Eq:l function}. Hence, $\smash{\TCD_q^e(K,N)}$ can be used to \emph{define} the meaning of the ``Ricci curvature bounded from below by $K$ and dimension bounded from above by $N$'', even though in this generality, neither a Ricci tensor nor a meaningful notion of dimension exist per se. The shape of  \cref{Def:TCD} itself is inspired by Erbar--Kuwada--Sturm \cite{erbar-kuwada-sturm2015}. Versions of \cref{Def:TCD} in line with approaches of Sturm \cite{sturm2006-i,sturm2006-ii} and Bacher--Sturm \cite{bacher-sturm2010}  can be found in Braun \cite{braun2023-good,braun2023-renyi}.

\MB{The precise way to interpret bounds on the Ricci tensor  here and henceforth is as follows. Given $K\in\R$, we say ``$\Ric \geq K$ holds in all timelike directions'' if 
\begin{align}\label{Eq:Ricvv}
\Ric(v,v)\geq K\,\vert v\vert^2
\end{align}
for every timelike vector $v\in T\mms$. (Recall our signature  convention $+,-,\dots,-$ for the Lorentzian metric from \cref{Def:Lormet}.) We define lower bounds on more general $2$-tensors in analogous way.}

\begin{theorem}[TCD condition and timelike Ricci curvature \cite{mccann2020,mondino-suhr2022}]\label{Th:EquivTCD} Suppose $(\mms,g)$ is a globally hyperbolic spacetime. Given any $K\in\R$ and any nonzero number $q$ less than one, the inequality $\Ric\geq K$ holds in all timelike directions if and only if $(\mms,g,\vol)$ satisfies the $\smash{\TCD_q^e(K,\dim\mms)}$ condition.
\end{theorem}

\begin{remark}[Bakry--Émery spacetimes] As indicated in the introduction, if one changes the reference measure from $\vol$ to e.g.~a weighted measure $\meas$ of the form $\rme^V\,\vol$, where $V$ is a smooth function on $\mms$, the conclusion of \cref{Th:EquivTCD} changes correspondingly. Given any $N \in [\dim\mms,\infty)$, it then asserts the equivalence of the $\smash{\TCD_q^e(K,N)}$ condition of the measured spacetime $(\mms,g,\meas)$ to lower boundedness of the following \smash{Bakry--Émery--Ricci}  tensor in all timelike directions by $K$:
\begin{align*}
\Ric^{N,V}:= \Ric + \Hess\,V - \frac{1}{N-\dim\mms}\,\rmd V\otimes\rmd V.
\end{align*}
Such spacetimes with weighted reference measures, so-called  \emph{\smash{Bakry--Émery} space\-times}, are studied e.g.~by Case \cite{case2010} and Woolgar--Wylie \cite{woolgar-wylie2016,woolgar-wylie2018}.\hfill$\blacksquare$
\end{remark}

With appropriate modifications, all results presented in the sequel cover the weighted case. To simplify the presentation, we  only state them in the unweighted case with reference measure $\vol$.

For later use, we report the following weaker pathwise formulation of the TCD condition (or, more precisely, the so-called timelike measure contraction property, introduced by  Cavalletti--Mondino \cite{cavalletti-mondino2020} and studied further by Braun \cite{braun2023-renyi}).

\begin{theorem}[Essential semiconcavity of transport densities \cite{braun2023-renyi,braun-ohta2024}]\label{Th:Essconc} Assume the measured spacetime $(\mms,g,\vol)$ satisfies $\Ric\geq K$ in all timelike directions, where $K\in\R$. Let $o\in \mms$ and define $\mu_0\in\Prob(\mms)$ by $\smash{\mu_0 := \delta_o}$.  Moreover, let $\mu_1\in\Prob(\mms)$ be compactly supported with $\smash{\supp\mu_1 \subset I^+(o)}$. Then the $\ell_q$-optimal dynamical coupling $\bdpi$ lifting the unique $\ell_q$-geodesic from $\mu_0$ to $\mu_1$ obeys the following bound between the respective $\vol$-densities for $\bdpi$-a.e.~$\gamma$ and every $t\in (0,1]$:
\begin{align*}
\rho_t(\gamma_t)^{-1/\dim\mms} &\geq \tau_{K,\dim\mms}^{(t)}(l(o,\gamma_1))\,\rho_1(\gamma_1)^{-1/\dim\mms}.
\end{align*} 
\end{theorem}

\subsection{D'Alembert comparison with timelike cut locus}\label{Sub:dAlembert comparison} Now we review the d'Alembert comparison theorem for Lorentz distance functions to a point (and suitable powers thereof) by Beran--Braun--Calisti--Gigli--McCann--Ohanyan--Rott--Sämann \cite{beran-braun-calisti-gigli-mccann-ohanyan-rott-samann+-},  established in a weak  fashion. It was preceded and inspired by an alike  Laplace comparison theorem on metric measure spaces with synthetic Ricci curvature bounds by  Gigli \cite{gigli2015}. We stress again that all comparison theorems on  spacetimes  before \cite{beran-braun-calisti-gigli-mccann-ohanyan-rott-samann+-} mentioned in the introduction were established \emph{outside} the future timelike cut locus. 

Fix a point $o\in\mms$. We first derive the d'Alembert comparison theorem for the power $\smash{\u_q := l_o^q/q}$. Here, $q$ is chosen as the  dual  exponent to $p$. Considering  this power of the Lorentz distance function instead of $l_o$ is well-motivated from the following simple observation, which links $\u_q$ to the $\ell_q$-optimal transport problem. 

\begin{proposition}[Kantorovich potentials for Dirac initial points]\label{Pr:KantorovDirac} Let $o\in\mms$ and define $\mu\in\Prob(\mms)$ by $\mu := \delta_o$. Furthermore, let $\nu\in\Prob(\mms)$ have compact support in $\smash{I^+(o)}$. Then $\u_q$ is a strong Kantorovich potential relative to $(\{o\},\supp\nu)$  according to \cref{Def:Strong Kantorovich}.
\end{proposition}

In particular, as $q$ is the conjugate exponent to $p$, (the time-reversal of)  \cref{Th:Optimal} implies the $\ell_q$-optimal transport from the previous proposition is displaced along the flow  $\Phi^q$ of the $p$-gradient of $\u_q$. On the other hand, this transport sees the timelike Ricci curvature by  the inequality from \cref{Th:Essconc}. How do  these two observations relate to the integration by parts formula for the distributional $p$-d'Alembertian from \cref{Def:Distrdalem}? Given $o\in\mms$, let us consider the quantity $\smash{\rmd\psi(\nabla \u_q)\,\vert\rmd \u_q\vert_*^{p-2}}$, where $\psi\in C_\comp^\infty(I^+(o))$. It can be computed in two different ways. 
\begin{itemize}
\item On the one hand, on the smoothness set $I^+(o)\setminus\TCut^+(o)$ of $l_o$, cf.~\cref{Th:Enhanced}, one would naturally use the formula
\begin{align}\label{Eq:bbb}
\rmd \psi(\nabla \u_q)\,\big\vert\rmd \u_q\big\vert_*^{p-2} = \lim_{\delta\to 0+} \frac{\big\vert\rmd(\u_q + \delta \psi)\big\vert_*^p - \big\vert\rmd \u_q\big\vert_*^p}{p\delta}.
\end{align}
The right-hand side corresponds to a \emph{vertical} (or \textit{outer}) derivative, as the perturbation occurs at the level of the functional variable. 
\item On the other hand, by the ODE for $\Phi^q$  \eqref{Eq:bbb} relates to a \textit{horizontal} (or \textit{inner}) derivative, where the perturbation occurs in the argument. 
\end{itemize}

\begin{lemma}[Vertical vs.~horizontal differentiation]\label{Le:VerthorizAbstract} Let $o\in\mms$. Let $\Phi^q$ designate the flow of the negative $p$-gradient of  $\u_q$. Then for every $\smash{\psi\in C_\comp^\infty(I^+(o))}$ and every $\smash{y\in I^+(o)\setminus \TCut^+(o)}$ we have 
\begin{align*}
\rmd \psi(\nabla \u_q)\,\big\vert\rmd \u_q\big\vert_*^{p-2} \circ\Phi^q_1(y) = \frac{\rmd}{\rmd s}\Big\vert_1 \psi \circ\Phi^q_s(y).
\end{align*}
\end{lemma}

To formulate the comparison results below, given any $K\in\R$  and a number $N$ larger than one, we define $\smash{\pi_{K,N}:= \sqrt{(N-1)/K}}$ if $K$ is positive and $\smash{\pi_{K,N}:= \infty}$ otherwise. In the first case, $\smash{\pi_{K,\dim\mms}}$ is the sharp Bonnet--Myers $l$-diameter bound provided $(\mms,g,\vol)$ satisfies $\Ric\geq K$ in all timelike directions, cf.~e.g.~Cavalletti--Mondino \cite{cavalletti-mondino2020}. For $\theta\in (0,\pi_{K,N})$ we then set
\begin{align}\label{Eq:Tfunction}
\sfT_{K,N}(\theta) := \begin{cases} \displaystyle\frac{1}{N} + \frac{\theta}{N}\sqrt{K(N-1)}\cot\!\Big[\theta\sqrt{\frac{K}{N-1}}\Big] & \textnormal{if } K>0,\\
1 & \textnormal{if } K=0,\\
\displaystyle\frac{1}{N}+\frac{\theta}{N}\sqrt{-K(N-1)}\coth\!\Big[\theta\frac{-K}{N-1}\Big] & \textnormal{otherwise}.
\end{cases}
\end{align}
Note that $\sfT_{K,N}(\theta)$ is the derivative of $\smash{r\mapsto \tau_{K,N}^{(r)}(\theta)}$ at one.

\begin{theorem}[D'Alembert comparison I {\cite{beran-braun-calisti-gigli-mccann-ohanyan-rott-samann+-}}]\label{Th:DAlembertcomparison} Let $(\mms,g,\vol)$ constitute  a globally hyperbolic measured spacetime satisfying $\Ric\geq K$ in all timelike directions, where $K\in\R$. Let $o\in\mms$ and let $p$ and $q$ be mutually conjugate nonzero exponents less than one. Then for every non\-negative function $\smash{\varphi\in \Lip_\comp(I^+(o))}$,
\begin{align*}
-\int_\mms\rmd\varphi(\nabla\u_q)\,\big\vert\rmd\u_q\big\vert^{p-2}_*\d\vol \leq \dim\mms \int_\mms \varphi\,\sfT_{K,\dim\mms}\circ l_o \d\vol.
\end{align*} 
\end{theorem}

We stress again the support of the test function $\varphi$ above does not need to be disjoint from the future timelike cut locus of $o$. 

Why is \cref{Th:DAlembertcomparison} entitled ``d'Alembert comparison''? A formal integration by parts on the left-hand side (say, if $K$ vanishes) taking into account \eqref{Eq:div def} makes it evident that \cref{Th:DAlembertcomparison} is a comparison theorem of the form ``$\Box_p\u_q\leq \dim\mms$''. However, we do not know yet if $\u_q$ has a $p$-d'Alembertian, in a distributional sense or otherwise. This will in fact be a \emph{consequence} of \cref{Th:DAlembertcomparison}, cf.~\cref{Th:Abstract existence} for an abstract or \cref{Th:Power} for a constructive argument.

\begin{proof}[Proof sketch of \cref{Th:DAlembertcomparison}] To simplify the presentation, we  only consider the case when $K$ vanishes, referring to \cite{beran-braun-calisti-gigli-mccann-ohanyan-rott-samann+-} for the general case. 

Define $\mu_0,\mu_1\in\Prob(\mms)$ by $\smash{\mu_0:=\delta_o}$ and $\mu_1:= (c\,\varphi)^{\dim\mms/(\dim\mms-1)}\,\vol$, where $c$ is an appropriate  normalization constant. \cref{Pr:KantorovDirac} makes \cref{Th:Optimal} applicable. First, this entails $\mu_1$ is concentrated on the smoothness set $\smash{I^+(o)\setminus \TCut^+(o)}$ from \cref{Th:Enhanced}. Second, consider the $\vol$-a.e. (therefore  $\mu_1$-a.e.) well-defined assignment $\sfF^q \colon I^+(o)\setminus \TCut^+(o)\to C([0,1];\mms)$ sending $\smash{y\in I^+(o)\setminus \TCut^+(o)}$ to the unique $l$-geodesic $\smash{\Phi_\cdot^q}$ from $o$ to $y$. Then $\smash{\bdpi := \sfF^q_\push\mu_1}$ is the unique $\smash{\ell_q}$-optimal dynamical coupling from $\mu_0$ to $\mu_1$. Third, the $t$-slice $\mu_t := (\eval_t)_\push\bdpi$ of $\bdpi$ is $\vol$-absolutely continuous and compactly supported for every $t\in (0,1]$. 

For such $t$, let $\rho_t$ denote the $\vol$-density of $\mu_t$. Observe that
\begin{align}\label{Eq:Prev}
-\int \rho_t(\gamma_t)^{-1/\dim\mms}\d\bdpi(\gamma) = -\int_\mms \rho_t^{-1/\dim\mms}\d\mu_t = -\int_\mms \rho_t^{1-1/\dim\mms}\d\vol.
\end{align} 
This is the $\dim\mms$-Rényi entropy, studied in conjunction with time\-like curvature-dimension bounds by Braun \cite{braun2023-renyi} inspired from Sturm's work \cite{sturm2006-ii}, of the $t$-slice $\mu_t$. Integrating the estimate from \cref{Th:Essconc} with respect to $\bdpi$ and using  \eqref{Eq:Prev},
\begin{align*}
-\int_\mms \rho_{1-t}^{1-1/\dim\mms}\d\vol \leq -(1-t) \int_\mms \rho_1^{1-1/\dim\mms}\d\vol.
\end{align*}
Rearranging terms, dividing by $t$, and sending $t$ to one,
\begin{align*}
\limsup_{t\to 1-} \int_\mms \frac{\rho_1^{1-1/\dim\mms} - \rho_{1-t}^{1-1/\dim\mms}}{t}\d\vol \leq \int_\mms\rho_1^{1-1/\dim\mms}\d\vol.
\end{align*}

By our choice of $\mu_1$, the right-hand side becomes
\begin{align*}
\int_\mms \rho_1^{1-1/\dim\mms}\d\vol = c\int_\mms \varphi\d\vol.
\end{align*}

On the other hand, by concavity of $\smash{r\mapsto r^{1-1/\dim\mms}}$ on $\smash{\R_+}$ and as in \eqref{Eq:Prev},
\begin{align*}
&\limsup_{t\to 1-}\int_\mms \frac{\rho_1^{1-1/\dim\mms}-\rho_{1-t}^{1-1/\dim\mms}}{t}\d\vol\\
&\qquad\qquad \geq \frac{\dim\mms-1}{\dim\mms}\,\limsup_{t\to 1-} \int_\mms \rho_1^{-1/\dim\mms}\,\frac{\rho_1-\rho_{1-t}}{t}\d\vol\\
&\qquad\qquad = \frac{\dim\mms-1}{\dim\mms}\,\limsup_{t\to 1-}\int  \frac{\rho_1(\gamma_1)^{-1/\dim\mms} - \rho_1(\gamma_{1-t})^{-1/\dim\mms}}{t}\d\bdpi(\gamma).
\end{align*}
A formal application of \cref{Le:VerthorizAbstract} with $\psi$ replaced by $\rho_1^{-1/\dim\mms}$ yields
\begin{align*}
&\limsup_{t \to 1-} \int \frac{\rho_1(\gamma_1)^{-1/\dim\mms} - \rho_1(\gamma_{1-t})^{-1/\dim\mms}}{t}\d\bdpi(\gamma)\\
&\qquad\qquad = \int \rmd\big[\rho_1^{-1/\dim\mms}\big](\nabla \u_q)\,\big\vert\rmd \u_q\big\vert_*^{p-2}(\gamma_1)\d\bdpi(\gamma)\\
&\qquad\qquad = \int_\mms \rmd\big[\rho_1^{-1/\dim\mms}\big](\nabla\u_q)\,\big\vert\rmd\u_q\big\vert_*^{p-2}\,\rho_1\d\vol.
\end{align*}
Finally, using the chain rule and the definition of $\rho_1$ yields
\begin{align*}
\rmd\big[\rho_1^{-1/\dim\mms}\big](\nabla\u_q)\,\big\vert\rmd\u_q\big\vert_*^{p-2}\,\rho_1 = -\frac{c}{\dim\mms-1}\,\rmd\varphi(\nabla\u_q)\,\big\vert\rmd\u_q\big\vert_*^{p-2}\quad\mu_1 \textnormal{-a.e.}
\end{align*}
Summarizing terms yields the desired estimate.
\end{proof}

This proof sketch clarifies that with analogous arguments,  previously unknown comparison results for the $p$-d'Alembertian of more general Kantorovich potentials for the $\ell_q$-optimal transport problem can be shown. We refer the reader to Beran--Braun--Calisti--Gigli--McCann--Ohanyan--Rott--Sämann \cite{beran-braun-calisti-gigli-mccann-ohanyan-rott-samann+-} for details.

The following genuine d'Alembert comparison theorem for the Lorentz distance function from a point $o\in\mms$ is now a direct  consequence of \cref{Th:DAlembertcomparison} and the chain rule. Since  $\smash{\vert\rmd l_o\vert_*^{p-2}=1}$ on $\smash{I^+(o)\setminus \TCut^+(o)}$ by  \cref{Th:Enhanced}, it also controls the classical d'Alembertian $\Box l_o$ within the future timelike cut locus of $o$ in an appropriate distributional way, cf.~\cref{Re:Indep!}.

\begin{theorem}[D'Alembert comparison II \cite{beran-braun-calisti-gigli-mccann-ohanyan-rott-samann+-}]\label{Th:DAlembertcomparisonII} Let $(\mms,g,\vol)$ designate a globally hyperbolic measured spacetime with $\Ric\geq K$ in all timelike directions, where $K\in\R$. Let $o\in\mms$. Then for every nonnegative  $\smash{\varphi\in\Lip_\comp(I^+(o))}$, 
\begin{align*}
-\int_\mms \rmd\varphi(\nabla l_o)\,\big\vert\rmd l_o\big\vert_*^{p-2}\d\vol \leq \int_\mms \varphi\,\frac{(\dim\mms)\,\sfT_{K,\dim\mms}\circ l_o-1}{l_o}\d\vol.
\end{align*}
\end{theorem}

One notable application of this result is a drastic simplification of the classical proof of the Lorentzian splitting theorem by Eschenburg \cite{eschenburg1988}, Galloway \cite{galloway1989-splitting}, and Newman \cite{newman1990} executed by Braun--Gigli--McCann--Ohanyan--Sämann \cite{braun-gigli-mccann-ohanyan-samann+}. A similar strategy leads to a version in progress by Braun--Gigli--McCann--Ohanyan--Sämann for Lorentzian metrics of low regularity \cite{braun-gigli-mccann-ohanyan-samann++} as well as  a splitting result in Finsler spacetime geometry by Caponio--Ohanyan--Ohta \cite{caponio-ohanyan-ohta2024+} extending  Lu--Minguzzi--Ohta \cite{lu-minguzzi-ohta2023-splitting}. We refer the interested reader to the review of McCann \cite{mccann+}.

\subsection{From d'Alembert comparison to existence} Recall from \cref{Re:Uniquenessdalem} that in the standing spacetime framework, the distributional $p$-d'Alembertian is always unique (if existent) and it satisfies the certifying integration by parts formula from \cref{Def:Distrdalem} with an equality. Therefore, it  remains to discuss its existence. This is where the above comparison theorems enter. Roughly speaking, they produce a nonnegative Radon functional, to which the Riesz--Markov--Kakutani representation theorem associates a Radon measure. Sufficient conditions for the applicability of this representation result were given by Beran--Braun--Calisti--Gigli--McCann--Ohanyan--Rott--Sämann \cite{beran-braun-calisti-gigli-mccann-ohanyan-rott-samann+-} following an analogous existence result of Gigli \cite{gigli2015} in Riemannian signature. An example of such a sufficient condition on abstract metric measure spacetimes is the so-called topological anti-Lipschitzness of Kunzinger--Steinbauer \cite{kunzinger-steinbauer2022} inspired by Sormani--Vega \cite{sormani-vega2016}, as verified by Braun \cite{braun2024+}.

For technical reasons pertaining to the degeneracy of the quantity from \eqref{Eq:Tfunction} at the sharp Bonnet--Myers $l$-diameter bound $\smash{\pi\sqrt{(\dim\mms-1)/K}}$ if $K$ is positive (cf. e.g. Cavalletti--Mondino \cite{cavalletti-mondino2020}), we define the open subset $\smash{I_{K,\dim\mms}^+(o)}$ of $I^+(o)$ by $\smash{I^+(o) \cap \{l_o < \pi_{K,\dim\mms}\}}$ if $K$ is positive, where $\pi_{K,\dim\mms}$ is from before \cref{Th:DAlembertcomparison}, and simply as $I^+(o)$ otherwise.

Recall the definition of the function $\u_q$ before \cref{Pr:KantorovDirac}, where $q$ denotes the dual exponent to $p$.

\begin{theorem}[Distributional $p$-d'Alembertian, existence I \cite{beran-braun-calisti-gigli-mccann-ohanyan-rott-samann+-}]\label{Th:Abstract existence} Let $(\mms,g,\vol)$ be a globally hyperbolic measured spacetime with $\Ric\geq K$ in all timelike directions, where $K\in\R$.  Let $o\in\mms$ be a given point and let $p$ and $q$ be mutually conjugate nonzero exponents less than one. Then the set $\smash{\Dom(\BOX_p \u_q\mres I_{K,\dim\mms}^+(o))}$ is nonempty. More strongly, there exists a unique Radon functional $T$ over $\smash{I^+_{K,\dim\mms}(o)}$ which is given as the difference of two Radon measures thereon and which satisfies the following integration by parts formula for every $\smash{\varphi\in\Lip_\comp(I_{K,\dim\mms}^+(o))}$:
\begin{align*}
\int_\mms\rmd\varphi(\nabla\u_q)\,\big\vert\rmd\u_q\big\vert_*^{p-2}\d\vol = -T(\varphi).
\end{align*}
\end{theorem}

\begin{proof}[Proof sketch] One first shows the map $Q$ defined on $\smash{\Lip_\comp(I_{K,\dim\mms}^+(o))}$ by
\begin{align*}
Q(\varphi) := \int_\mms \rmd\varphi(\nabla\u_q)\,\big\vert\rmd\u_q\big\vert_*^{p-2}\d\vol + \dim\mms\int_\mms\varphi\,\sfT_{K,\dim\mms}\circ l_o\d\vol
\end{align*}
constitutes a Radon functional. It  has a nonrelabeled extension to $\smash{C_\comp(I_{K,\dim\mms}^+(o))}$. Moreover, by \cref{Th:DAlembertcomparison} it is nonnegative. Thus, by the Riesz--Markov--Kakutani representation theorem there is a Radon measure $\mu$ on $\smash{I^+_{K,\dim\mms}(o)}$ such that $Q$ is given by integration against $\mu$. Then the map $T$ on $\smash{\Lip_\comp(I_{K,\dim\mms}^+(o))}$ given by
\begin{align*}
T(\varphi) := (\dim\mms)\,\sfT_{K,\dim\mms}\circ l_o\,\vol - \mu,
\end{align*}
is a Radon functional (as the difference of two Radon measures, where the Radon property is ensured by the definition of $I_{K,\dim\mms}(o)$). By construction, it satisfies the integration  by parts formula from \cref{Def:Distrdalem}.
\end{proof}

We stress again that by the previous proof, the distributional $p$-d'Alembertian of $\u_q$ is the difference of two Radon measures. It makes sense as a Radon functional but is in general not a signed measure.

\begin{theorem}[Distributional $p$-d'Alembertian, existence II \cite{beran-braun-calisti-gigli-mccann-ohanyan-rott-samann+-}] Let $(\mms,g,\vol)$ be a globally hyperbolic measured spacetime such that $\Ric\geq K$ in all timelike directions, where $K\in\R$. Let $o\in\mms$. Then the class $\smash{\Dom(\BOX_p l_o\mres I^+_{K,\dim\mms}(o))}$ is nonempty. More strongly, there exists a unique Radon functional $S$ over $\smash{I^+_{K,\dim\mms}(o)}$ which is given as the difference of two Radon measures thereon and which obeys the subsequent integration by parts formula for every $\smash{\varphi\in\Lip_\comp(I^+_{K,\dim\mms}(o))}$:
\begin{align*}
\int_\mms\rmd\varphi(\nabla l_o)\,\big\vert\rmd l_o\big\vert_*^{p-2}\d\vol = -S(\varphi).
\end{align*}
\end{theorem}

\section{Constructive existence and applications}

A closer inspection of the relation of optimal transport and gradient flow techniques already outlined in \cref{Sub:dAlembert comparison} actually leads to more precise information on the distributional $p$-d'Alembertian, namely exact representation formulas. This section surveys adjacent results recently obtained by Braun \cite{braun2024+}.

\subsection{Prerequisite. Localization} The guiding  principle  of the convex analytic technique of \emph{localization} is to simplify a high-dimensional problem by reducing it into a ``family'' of one-dimensional problems, where it is frequently easier to solve. (This heuristic explains the other commonly used  name \emph{needle decomposition} for this paradigm.) Rigorously, invoking  this technique amounts to employ a form of the disintegration theorem from probability theory in a specific way tailored to the  problem in question. We  make this precise in \cref{Th:Disint} below; for now,  let us specify that as inspired by Cavalletti--Mondino \cite{cavalletti-mondino2020-new},  the problem itself we eventually ``localize'' is the one of \emph{removing the derivative from the test function in order to get an integration by parts formula for the distributional $p$-d'Alembertian}. After this  reduction, the representation formulas of Braun \cite{braun2024+} we survey will --- very roughly speaking --- follow from the integration by parts formula on the real line.

The localization technique was pioneered by Payne--Weinberger \cite{payne-weinberger1960} in proving the optimal Euclidean Poincaré inequality. It was developed further by Gromov--Milman \cite{gromov-milman1987}, Lovász--Simonovits \cite{lovasz-simonovits1993}, and Kannan--Lóvasz--Simonovits \cite{kannan-lovasz-simonovits1995}. In the context of optimal transport, it was studied particularly by Sudakov \cite{sudakov1979}, Trudinger--Wang \cite{trudinger-wang2001},  Caffarelli--Feldman--McCann \cite{caffarelli-feldman-mccann2002}, and  Ambrosio \cite{ambrosio2003} (in Euclidean space) as well as Feldman--McCann \cite{feldman-mccann2002}, Bianchini--Cavalletti \cite{bianchini-cavalletti2013}, and Cavalletti \cite{cavalletti2012,cavalletti2014} (in more general spaces). On Riemannian manifolds, Klartag \cite{klartag2017} realized ambient Ricci curvature bounds are inherited by the needles, a result that was subsequently extended to metric measure spaces with synthetic Ricci curvature bounds à la Sturm \cite{sturm2006-i,sturm2006-ii} and Lott--Villani \cite{lott-villani2009} by Cavalletti--Mondino \cite{cavalletti-mondino2017-isoperimetric}. Both approaches entail an entire set of geometric and functional inequalities (see also Cavalletti--Mondino's sequel \cite{cavalletti-mondino2017-geometric}). The power this toolbox  has demonstrated in applications  cannot be given appropriate credit here, given the large literature; see e.g.~the review of Cavalletti \cite{cavalletti2017} or Villani's survey \cite{villani2019} in the Séminaire Bourbaki. Cornerstones are Cavalletti--Milman's  globalization theorem \cite{cavalletti-milman2021},  Cavalletti--Mondino's representation formulas for certain Laplacians \cite{cavalletti-mondino2020-new} (cf.~\cref{Sub:Exact} below), Ketterer's nonsmooth extension \cite{ketterer2020-heintze-karcher} of the Heintze--Karcher inequality from Riemannian geometry \cite{heintze-karcher1978} (cf.~\cref{Sub:More} below), and Burtscher--Ketterer--McCann--Woolgar's ``Riemannian'' Hawking-type singularity theorem \cite{burtscher-ketterer-mccann-woolgar2020} (cf.~\cref{Sub:Volume sing} below).

In Lorentzian signature, Cavalletti--Mondino \cite{cavalletti-mondino2020} pioneered needle decomposition techniques to prove, among other things, a synthetic Hawking singularity theorem for abstract TCD spaces, cf.~\cref{Th:SyntheticHawking}. Subsequent extensions were given by themselves \cite{cavalletti-mondino2024} (using a stronger version of \cref{Th:Essconc}) and Braun--McCann \cite{braun-mccann2023}. The  result we report here comes from \cite{cavalletti-mondino2024}.

Let $\mathfrak{M}(\mms)$ denote the class of all Borel measures on $\mms$. \MB{We recall \eqref{Eq:Ricvv} for the interpretation of lower boundedness of the Ricci curvature in all timelike directions.}

\begin{theorem}[Disintegration theorem \cite{cavalletti-mondino2024}]\label{Th:Disint} Let the globally hyperbolic  measured spacetime $(\mms,g,\vol)$ satisfy $\Ric \geq K$ in all timelike directions, where $K\in\R$. Let $\Sigma$ form an achronal, compact, spacelike hypersurface in $\mms$. Then there are a probability measure $\q$ on $\Sigma$ and a $\q$-measurable map $\vol_\cdot \colon \Sigma \to \mathfrak{M}(\mms)$ satisfying the following properties. 
\begin{enumerate}[label=\textnormal{\textcolor{black}{(}\roman*\textcolor{black}{)}}]
\item\label{La:DISINT} \textnormal{\textbf{Disintegration.}} We have
\begin{align*}
\vol \mres I^+(\Sigma) = \int_\Sigma \vol_\alpha\d\q(\alpha).
\end{align*}
\item \textnormal{\textbf{Strong consistency.}} For $\q$-a.e.~$\alpha\in\Sigma$, the measure $\vol_\alpha$ is concentrated on the negative gradient flow line $\Phi_\cdot(\alpha)$ of $l_\Sigma$ starting from $\alpha$ according to  \cref{Th:Enhanced}, abbreviated by $\mms_\alpha$.
\item \textnormal{\textbf{Local finiteness.}} For every compact subset $C$ of $\smash{I^+(\Sigma)}$,
\begin{align*}
\q\textnormal{-}\!\esssup_{\alpha\in \Sigma} \vol_\alpha[C] < \infty.
\end{align*}
\item \textnormal{\textbf{Curvature-dimension condition.}} Let $\vert\cdot\vert$ denote the usual absolute value  on the real line. For $\q$-a.e.~$\alpha\in \Sigma$, the topologically one-dimensional metric measure space $(\mms_\alpha,\vert\cdot-\cdot\vert,\vol_\alpha)$ satisfies the curvature-dimension condition $\CD(K,\dim\mms)$ in the sense of  Sturm \cite{sturm2006-i,sturm2006-ii} and Lott--Villani \cite{lott-villani2009}; more strongly, $\vol_\alpha$ is absolutely continuous with respect to the  Hausdorff measure $\Haus_\alpha^1$ of $(\mms_\alpha,\vert\cdot-\cdot\vert)$, and its density $h_\alpha$  satisfies   the following $(K,\dim\mms)$-convexity  inequality after Erbar--Kuwada--Sturm \cite{erbar-kuwada-sturm2015} for every $t_0,t_1\in \mms_\alpha$ with $t_0< t_1$ and every $\lambda\in[0,1]$:
\begin{align}\label{Eq:Semicon}
\begin{split}
&h_\alpha((1-\lambda)\,t_0 + \lambda\,t_1)^{1/(\dim\mms-1)}\\
&\qquad\qquad \geq \sigma_{K/(\dim\mms-1)}^{(1-\lambda)}(t_1-t_0)\,h_\alpha(t_0)^{1/(\dim\mms-1)}\\
&\qquad\qquad\qquad\qquad + \sigma_{K/(\dim\mms-1)}^{(\lambda)}(t_1-t_0)\,h_\alpha(t_1)^{1/(\dim\mms-1)}.
\end{split}
\end{align}
In particular, if $K$ vanishes, this translates into concavity of $\smash{h_\alpha^{1/(\dim\mms-1)}}$.
\end{enumerate}

Moreover, for $\q$ as above the disintegration is $\q$-essentially unique. This means if $\smash{\vol_\cdot'\colon \Sigma\to\mathfrak{M}(\mms)}$ is another map with the same  properties, then $\vol_\cdot = \vol_\cdot'$ $\q$-a.e.
\end{theorem}

The preceding negative gradient flow lines $\mms_\alpha$ of $l_\Sigma$ are also called \emph{rays},  where $\alpha\in\Sigma$, while the densities $h_\alpha$  are termed \emph{conditional densities}.

\begin{remark}[Identifications]\label{Re:Identifications} As the attentive  reader may have already noticed in the formulation of \cref{Th:Disint}, above and below we  tacitly identify the gradient flow trajectory $\mms_\alpha$ starting at $\alpha\in\Sigma$, which per se forms a subset of $\smash{I^+(\Sigma)\cup\Sigma}$, with the subinterval of $[0,\infty)$ on which the negative gradient flow $\Phi_\cdot(\alpha)$ of $l_\Sigma$  is defined. Correspondingly, while in the rigorous formulations of the results based on \cref{Th:Disint} surveyed in this article one should regard the conditional densities as being defined on a subset of $\smash{I^+(\Sigma)\cup\Sigma}$, in proofs it is frequently convenient to view them as real-to-real functions under the previous identification. \hfill$\blacksquare$
\end{remark}

For later use, given $\alpha\in \Sigma$ we will denote by $a_\alpha$ and $b_\alpha$ the initial and the final point of the gradient flow trajectory $\mms_\alpha$. The set $\{b_\alpha: \alpha\in \Sigma\}$ coincides with the future timelike cut locus $\smash{\TCut^+(\Sigma)}$ of $\Sigma$. If $b_\alpha$ does not exist for a given $\alpha\in\Sigma$, its contribution is disregarded throughout our subsequent discussion. 

Before collecting some basic properties of the disintegration, we  elaborate on \cref{Th:Disint} more informally. For illustrative purposes, although we assumed noncompactness of $\mms$, suppose  $\smash{\vol\mres I^+(\Sigma)}$ is a probability measure. Let $\sfF$ denote the footpoint projection map on $\Sigma$ from \cref{Le:Footpoint proj}. Then $\smash{\q := \sfF_\push[\vol\mres I^+(\Sigma)]}$ is a probability measure on $\Sigma$ and the disintegration theorem from probability theory, cf.~e.g.~\cite{fremlin2006}, implies the existence of a disintegration as in item \ref{La:DISINT} above. So what is  the key point of \cref{Th:Disint}? There are two crucial geometric addenda.
\begin{itemize}
\item Usually, the abstract disintegration theorem yields  for $\q$-a.e.~$\alpha\in\Sigma$, the conditional measure $\vol_\alpha$ is concentrated on $\sfF^{-1}(\{\alpha\})$. In our geometric situation, as easily realized from \cref{Le:Footpoint proj} this is precisely the timelike proper time parametrized maximizer starting at $\alpha$ radially to $\Sigma$.  (Here we apply the simplification of $\Sigma$ being a hypersurface. For instance, if $\Sigma$ is a singleton, all rays will emanate from the same point. Yet, even for general $\Sigma$, $\q$-a.e.~ray will be a negative gradient flow curve of $l_\Sigma$.)
\item The hypothesized timelike Ricci curvature bound of the ambient space transfers  into a synthetic Ricci curvature bound for  $\q$-a.e.~ray. This should not be surprising: the rays follow timelike maximizers in their tangential directions, but the Ricci curvature in these directions always vanishes by the geodesic equation. (This explains the factor $\dim\mms-1$ in \eqref{Eq:Semicon}.)
\end{itemize}

\begin{remark}[Role of $\q$] Typically, the existence of the probability measure $\q$ in the general version of  \cref{Th:Disint} is constructive, but the construction is rather abstract. In our simpler setting of \cref{Th:Disint}, it can still be built  explicitly as an appropriate density times the canonical surface measure of $\Sigma$, e.g.~\cite{cavalletti-mondino2020}*{Rem.~5.4} or \cite{cavalletti-mondino2022-review}*{§3.2}. This level of abstraction does not have geometric consequences, as $\q$ is merely used to label the rays; all relevant information is encoded in the conditional densities (which, in the setting of \cref{Th:Disint}, naturally relate to the Jacobian determinant of the negative gradient flow of $l_\Sigma$).

Moreover, in general $\q$ will be concentrated on an abstract subset of $\smash{I^+(\Sigma)\cup\Sigma}$ instead of $\Sigma$. Indeed, if $\Sigma$ was a singleton $\{o\}$ yet $\q = \delta_o$, it would clearly  not distinguish rays. The existence of exactly one timelike direction radially emanating  from $\Sigma$ is one benefit of our hypothesis of $\Sigma$ being a spacelike hypersurface.
\hfill$\blacksquare$
\end{remark}

\begin{remark}[Properties of the conditional densities]\label{Re:Propertiesconddens} In the context of \cref{Th:Disint}, for $\q$-a.e.~$\alpha\in \Sigma$ the semiconvexity of $h_\alpha$ and the full support of $\vol$ imply the following properties. The conditional density $h_\alpha$ is positive and locally Lipschitz continuous on the relative interior of $\mms_\alpha$. In particular, it is differentiable $\Haus_\alpha^1$-a.e.; thus, the logarithmic derivative $(\log h_\alpha)'$ exists $\Haus_\alpha^1$-a.e. Lastly, the a priori only $\Haus_\alpha^1$-measurable function $h_\alpha$ can be continuously extended to all of $\mms_\alpha$.\hfill$\blacksquare$
\end{remark}

\begin{remark}[Generalization of \cref{Th:Disint}] The formulation of \cref{Th:Disint} generalizes straightforwardly to abstract metric measure spacetimes $(\mms,l,\meas)$ satisfying $\smash{\TCD_p^e(K,N)}$, where $N \in (1,\infty)$ and $p\in (-\infty,1)\setminus \{0\}$. Here, one has to impose additionally that $(\mms,l,\meas)$ is \emph{timelike $p$-essentially nonbranching}, a  measure-theoretic concept introduced by Braun \cite{braun2023-renyi} (inspired by  notions for metric measure spaces by Rajala--Sturm \cite{rajala-sturm2014} and Akdemir--Cavalletti--Colinet--McCann--Santarcangelo \cite{akdemir-colinet-mccann-cavalletti-santarcangelo2021}) of timelike nonbranching as considered by Cavalletti--Mondino \cite{cavalletti-mondino2020}. (The value of $p$ effectively does not matter.) It prevents the ``separation'' of negative gradient flow lines of Lorentz distance functions. In this generality, all occurrences of $\vol$ and $\dim\mms$ in \cref{Th:Disint} have to be replaced by $\meas$ and $N$, respectively.

Deriving the shape of \cref{Th:Disint} for weighted measured spacetimes $(\mms,g,\meas)$, where $\meas  = \rme^V\,\vol$ for a smooth function $V$ on $\mms$, is easy. Indeed, let us  suppose $\smash{\vol\mres I^+(\Sigma)}$ admits a disintegration $\q$ according to \cref{Th:Disint}. Then
\begin{align*}
\meas \mres I^+(\Sigma) =  \rme^V \vol \mres I^+(\Sigma) = \int_\Sigma \rme^V\,\meas_\alpha\d\q(\alpha). 
\end{align*}
However, unless $V$ has certain concavity properties along $\q$-a.e.~ray (cf.~e.g.~\cite{erbar-kuwada-sturm2015}*{Lem.~2.10}), the conditional densities $\smash{\rme^V\,h_\alpha}$, where $\alpha\in\Sigma$, will not obey \eqref{Eq:Semicon}.\hfill $\blacksquare$
\end{remark}

\subsection{Exact representation formulas}\label{Sub:Exact} Now we review (special cases of the) exact representation formulas for various distributional $p$-d'Alembertians by  Braun \cite{braun2024+}. They were preceded and inspired  by counterparts   for various  Laplacians on metric measure spaces with synthetic Ricci curvature bounds by Cavalletti--Mondino \cite{cavalletti-mondino2020-new}.

To get started, we report the following more general form of \cref{Le:VerthorizAbstract}. The gradient flow picture enters through \cref{Th:Enhanced}.

\begin{lemma}[Extended vertical vs.~horizontal differentiation]\label{Le:Verthoriz} Let $\Sigma$ be a smooth, compact,  spacelike hypersurface in $\mms$. Let $\Phi$ denote the negative gradient flow of  $l_\Sigma$. Then for every $\smash{\varphi\in C_\comp^\infty(I^+(\Sigma))}$, every  initial point $\alpha\in \Sigma$, and every $t > 0$ such that $\Phi_t(\alpha)$ is defined yet not a future timelike cut point of $\alpha$, 
\begin{align*}
\rmd \varphi(\nabla l_\Sigma)\,\big\vert\rmd l_\Sigma\big\vert_*^{p-2} \circ \Phi_t(\alpha) = \frac{\rmd}{\rmd t} \varphi\circ\Phi_t(\alpha).
\end{align*}
\end{lemma}

Hence, along each negative gradient flow line of $l_\Sigma$, the objective of computing the term on the left-hand side translates into taking a one-dimensional derivative $\varphi'$. Together with \cref{Th:Disint}, we are in a  position to state the following.  \MB{We recall \eqref{Eq:Ricvv} for the interpretation of lower boundedness of the Ricci curvature in all timelike directions.}

\begin{theorem}[Exact distributional $p$-d'Alembertian  I {\cite{braun2024+}}]\label{Th:Repr} Assume the globally hyperbolic measured spacetime $(\mms,g,\vol)$ satisfies $\Ric\geq K$ in all timelike directions, where $K\in\R$. Let $\Sigma$ be an achronal, compact, spacelike hypersurface in $\mms$. Lastly, let $\q$ designate  a disintegration of $\smash{\vol\mres I^+(\Sigma)}$ according to \cref{Th:Disint}. Then the assignment $T$ defined on $\smash{C_\comp(I^+(\Sigma))}$ by the formula
\begin{align*}
T(\varphi) := \int_\Sigma\int_{\mms_\alpha} \varphi\,(\log h_\alpha)'\d\vol_\alpha\d\q(\alpha) - \int_\Sigma \big[\varphi\,h_\alpha\big]^{b_\alpha}\d\q(\alpha)
\end{align*}
defines the unique Radon functional belonging to $\smash{\BOX_p l_\Sigma \mres I^+(\Sigma)}$.
\end{theorem}

The hypothesized timelike Ricci curvature bound appears only implicitly in the conclusion of \cref{Th:Repr}, in particular through the regularity properties of the conditional densities from \cref{Re:Propertiesconddens}.

Of course, the actual form of $T$ above does not depend on $p$, cf.~\cref{Re:Indep!}.

\begin{proof}[Proof sketch of \cref{Th:Repr}] By employing  qualitative a priori estimates for the conditional densities, see e.g.~Cavalletti--Milman \cite{cavalletti-milman2021} and Cavalletti--Mondino \cite{cavalletti-mondino2020-new}, one shows the map $T$ in question is a Radon functional. 

Uniqueness follows from \cref{Re:Uniquenessdalem}.

To establish the claimed representation formula, assume first $\varphi$ is smooth and has compact support in $\smash{I^+(\Sigma)}$. Let $\alpha\in \Sigma$ obey all statements from \cref{Th:Disint}. To simplify the presentation, we assume the negative gradient flow trajectory $\mms_\alpha$ starting at $\alpha$ has a timelike future cut point $b_\alpha$. As $\varphi$ is supported away from $\Sigma$, by \cref{Th:Enhanced} its  nonrelabeled composition with the negative gradient flow $\Phi_\cdot(\alpha)$ of $l_\Sigma$ is smooth except   at the future timelike cut point $b_\alpha$, which may belong to the support of $\varphi$. Integrating the identity from \cref{Le:Verthoriz} with respect to $\vol_\alpha$ and using the disintegration \cref{Th:Disint} yields 
\begin{align*}
\int_{\mms_\alpha} \rmd\varphi(\nabla l_\Sigma)\,\big\vert\rmd l_\Sigma\big\vert_*^{p-2}\d\vol_\alpha = \int_{\mms_\alpha} \varphi' \d\vol_\alpha = \int_{\mms_\alpha} \varphi'\,h_\alpha\d\Haus_\alpha^1.
\end{align*}
Using the regularity properties of $h_\alpha$ stipulated in \cref{Re:Propertiesconddens}, we can integrate the last term by parts, yielding
\begin{align*}
\int_{\mms_\alpha}\varphi'\,h_\alpha\d\Haus_\alpha^1 &= -\int_{\mms_\alpha} \varphi\,h_\alpha'\d\Haus_\alpha^1 + \big[\varphi\,h_\alpha\big]_{a_\alpha}^{b_\alpha}\\
&= -\int_{\mms_\alpha} \varphi\,(\log h_\alpha)'\d\vol_\alpha + \big[\varphi\,h_\alpha\big]^{b_\alpha}.
\end{align*}
In the last identity, we used $\varphi$ vanishes at $a_\alpha$, since its support is distinct from $\Sigma$.  
Integrating the obtained identities with respect to $\q$ and using the disintegration \cref{Th:Disint} once again, we arrive at the desired integration by parts formula. 

Given $T$ is a Radon functional, the extension beyond smooth functions is done by a customary approximation argument.
\end{proof}

We collect several observations about \cref{Th:Repr} and its proof.
\begin{itemize}
\item Unlike the abstract  existence theorem for the distributional $p$-d'Alembertian from \cref{Th:Abstract existence}, the proof of \cref{Th:Repr} is constructive and provides a precise representation formula.
\item The representation formula decomposes into a $\vol$-absolutely continuous contribution and a $\vol$-singular one. The $\vol$-singular contribution is concentrated on the set of future timelike cut points $\{b_\alpha: \alpha\in \Sigma\}$. Lastly, as $h_\alpha$ is nonnegative, if the test function $\varphi$ is nonnegative, the $\vol$-singular contribution is \emph{nonpositive}. This confirms and sharpens the qualitative nonnegativity of the distributional $p$-d'Alembertian on future timelike cut loci indicated by the d'Alembert comparison  \cref{Th:DAlembertcomparison}.
\item Comparison results can be directly inferred from simple estimates for the logarithmic derivative, cf.~\cref{Le:Logarithmicderivative} plus  \cref{Th:CompI,Th:CompII}.
\item A further comment on the curvature hypothesis in \cref{Th:Repr}. In the Riemannian context, the  representation formulas for the Laplacian of distance functions by Cavalletti--Mondino \cite{cavalletti-mondino2020-new} always hold locally, as the Ricci curvature is always bounded from below on compact sets. In the Lorentzian case, the latter is no more true. In fact, by  McCann \cite{mccann2023-null} we know locally uniform lower boundedness of the timelike Ricci curvature (which is still covered by the surveyed results of Braun \cite{braun2024+}) is equivalent to the null energy condition of Penrose \cite{penrose1965}.
\end{itemize}

It is well-known in Riemannian geometry that the squared distance function (say, to a point $o$) admits better regularity properties than its root $\met_o$, cf.~e.g.~Cheeger--Ebin \cite{cheeger-ebin1975}. Notably, $\met_o^2$ is smooth at the origin, unlike $\met_o$. This is reflected in the Laplace comparison theorems they satisfy e.g.~under nonnegative Ricci curvature: while the inequality  $\Delta \met_o \leq (\dim\mms-1)/\met_o$ --- interpreted in any suitable sense --- entails no uniform control of the Laplacian of $l_o$ near $o$, the bound $\Delta \met_o^2 \leq 2\dim\mms$ does. We invite the reader to also compare \cref{Th:DAlembertcomparison,Th:DAlembertcomparisonII}.

In our spacetime setting, given the shape of the $p$-d'Alembertian, the correct power to raise the involved  Lorentz distance functions to is the dual exponent $q$ of  $p$. For the distributional $p$-d'Alembertian of the resulting functions, the exact representation formulas established by  Braun \cite{braun2024+} read as follows.

\begin{theorem}[Exact distributional $p$-d'Alembertian  II \cite{braun2024+}]\label{Th:Power} Suppose the globally hyperbolic measured spacetime $(\mms,g,\vol)$ satisfies $\Ric\geq K$ in all timelike directions, where $K\in\R$. Let $\Sigma$ be an achronal, compact, spacelike hypersurface in $\mms$. Let $\q$ be a disintegration of $\smash{\vol\mres I^+(\Sigma)}$ after  \cref{Th:Disint}. Furthermore, let $p$ and $q$ be mutually conjugate nonzero exponents less than one and set $\smash{\u_q := l_\Sigma^q/q}$. Then the map $T^q$  on $\smash{C_\comp(I^+(\Sigma))}$ given  by 
\begin{align*}
T^q(\varphi) := \int_\Sigma\int_{\mms_\alpha} \varphi\,\big[1+ l_\Sigma\,(\log h_\alpha)'\big]\d\vol_\alpha\d\q(\alpha) - \int_\Sigma\big[\varphi\,l_\Sigma\,h_\alpha\big]^{b_\alpha}\d\q(\alpha)
\end{align*}
defines the unique Radon functional belonging to $\smash{\BOX_p \u_q\mres I^+(\Sigma)}$.
\end{theorem}

This result follows by a combination of the chain rule in the identity defining the distributional $p$-d'Alembertian from $\u_q$ from \cref{Def:Distrdalem} with \cref{Th:Repr}. By duality, cancellations lead to the ``invisibility'' of $p$ and $q$ in the above formula.

\subsection{More comparison theory}\label{Sub:More} Now we outline how several comparison results derive from disintegration techniques. In both surveyed topics, the strategy is the same: by basic manipulations, the semiconvexity inequality \eqref{Eq:Semicon} yields an estimate for the conditional densities which globalizes by  \cref{Th:Disint}. The results were obtained in a nonsmooth fashion by Braun \cite{braun2024+}. We also point out that with related disintegration techniques, timelike geometric and isoperimetric inequalities were obtained by Cavalletti--Mondino \cite{cavalletti-mondino2020,cavalletti-mondino2024}.

\subsubsection{D'Alembert comparison} The first basic prerequisite is reminiscent of the fact that the  derivative of a convex function is nondecreasing. In the context of localization, such density bounds have been first derived by Cavalletti \cite{cavalletti2014-monge}.

Given any $K\in\R$ and a real number $N$ larger than one, define $\pi_{K,N}$ as before \cref{Th:DAlembertcomparison}. For $\theta\in (0,\pi_{K,N})$ we then set
\begin{align*}
\sfC_{K,N}(\theta) := \begin{cases}
\displaystyle\sqrt{\frac{K}{N-1}}\,\frac{\displaystyle\cos\!\Big[\sqrt{\frac{K}{N-1}}\,\theta\Big]}{\displaystyle\sin\!\Big[\sqrt{\frac{K}{N-1}}\,\theta\Big]} & \textnormal{if }K>0,\\
\displaystyle\frac{1}{\theta} & \textnormal{if }K=0,\\
\displaystyle\sqrt{\frac{-K}{N-1}}\,\frac{\displaystyle\cosh\!\Big[\sqrt{\frac{-K}{N-1}}\,\theta\Big]}{\displaystyle\sinh\!\Big[\sqrt{\frac{-K}{N-1}}\,\theta\Big]} & \textnormal{otherwise}.
\end{cases}
\end{align*}
Note that $\sfC_{K,N}$ summarizes to $f'/f$, where $f$ solves $f''+ K\,f/(N-1)=0$ on $\R_+$ with initial conditions $f(0)=0$ and $f'(0) =1$.

\begin{lemma}[Logarithmic derivative estimate \cite{cavalletti2014-monge}]\label{Le:Logarithmicderivative} Let $\Sigma$ designate an achronal, compact, spacelike hypersurface in $\mms$. Suppose $\q$ is a disintegration of $\vol\mres I^+(\Sigma)$ obeying the conclusions from \cref{Th:Disint}. Then for $\q$-a.e.~$\alpha\in \Sigma$, the estimate \eqref{Eq:Semicon} implies at every differentiability point of $h_\alpha$ within $\mms_\alpha$ that
\begin{align*}
&-(\dim\mms-1)\,\sfC_{K,\dim\mms}\circ (\Len_l\,\mms_\alpha - l_\Sigma) \leq (\log h_\alpha)' \leq (\dim\mms-1)\,\sfC_{K,\dim\mms}\circ l_\Sigma,
\end{align*}
where $\smash{\Len_l\,\mms_\alpha}$ denotes the $l$-length of the ray $\mms_\alpha$.
\end{lemma}

If $\Len_l\,\mms_\alpha$ is infinite for a given $\alpha\in\Sigma$, the left-hand side will be understood as the limit of $\sfC_{K,\dim\mms}(\theta)$ as $\theta\to\infty$.

Given two Radon functionals $S$ and $T$ over an open subset $U$ of $\mms$, we write $S\leq T$ if $S(\varphi)\leq T(\varphi)$ for every nonnegative $\smash{\varphi\in\Lip_\comp(U)}$. Furthermore, as before \cref{Th:Abstract existence} we have to exclude the occurrence of regions with maximal $l$-distance from $\Sigma$  if the timelike Ricci curvature becomes uniformly positive (to ensure the appearing right-hand sides are still Radon functionals). To this aim, given $K\in\R$ and a nonzero number $N$ larger than one, we define the open subset $\smash{I^+_{K,N}(\Sigma)}$ of $\smash{I^+(\Sigma)}$ by $\smash{I^+(\Sigma) \cap \{l_\Sigma < \pi_{K,N}\}}$ if $K$ is positive and $I^+(\Sigma)$ otherwise.

\begin{theorem}[Extended d'Alembert comparison I \cite{braun2024+}]\label{Th:CompI} Assume $(\mms,g,\vol)$ is a globally hyperbolic measured spacetime with $\Ric\geq K$ in all timelike directions, where $K\in\R$. Let $\Sigma$ be an achronal, compact, spacelike hypersurface in $\mms$. Moreover, let $\q$ designate a disintegration of $\smash{\vol\mres I^+(\Sigma)}$ according to \cref{Th:Disint}. Then the Radon functional $T$ from \cref{Th:Repr} satisfies
\begin{align*}
T\big\vert_{\Lip_\comp(I^+_{K,\dim\mms}(\Sigma))}\leq (\dim\mms-1)\,\sfC_{K,\dim\mms} \circ l_\Sigma\,\vol\mres I^+_{K,\dim\mms}(\Sigma).
\end{align*}

Moreover, the first summand $T^\ll$ of $T$ obeys
\begin{align*}
T^\ll\big\vert_{\Lip_\comp(I^+_{K,\dim\mms}(\Sigma))} &\geq -(\dim\mms-1)\int_\Sigma \textcolor{black}{\d\q(\alpha)\,\vol_\alpha\mres I^+_{K,\dim\mms}(\Sigma)}\\
&\qquad\qquad \textcolor{black}{\times\, \sfC_{K,\dim\mms} \circ (\Len_l\,\mms_\alpha - l_\Sigma)}.
\end{align*}
\end{theorem}

These bounds can be read off from the representation formula of \cref{Th:Repr} using \cref{Le:Logarithmicderivative}. Note the remarkable fact that \cref{Th:CompI} does not only give an upper bound on the entire distributional $p$-d'Alembertian  (as the contribution of the future timelike cut locus is nonpositive), but also a lower estimate on the $\vol$-absolutely continuous part.

\begin{theorem}[Extended d'Alembert comparison II \cite{braun2024+}]\label{Th:CompII} Suppose $(\mms,g,\vol)$ is a globally hyperbolic measured spacetime such that $\Ric\geq K$ in all timelike directions, where $K\in\R$. Let $\Sigma$ be an achronal, compact, spacelike hypersurface in $\mms$. Let $\q$ be as in \cref{Th:Disint}. Lastly, let $p$ and $q$ be nonzero mutually conjugate exponents less than one. Then the Radon functional $T^q$ from \cref{Th:Power} satisfies
\begin{align*}
T^q\big\vert_{\Lip_\comp(I^+_{K,\dim\mms}(\Sigma))} &\leq \vol\mres I^+_{K,\dim\mms}(\Sigma)\\ 
&\qquad\qquad + (\dim\mms-1)\,l_\Sigma\,\sfC_{K,\dim\mms}\circ l_\Sigma\,\vol\mres I^+_{K,\dim\mms}(\Sigma).
\end{align*}

In addition, the first summand ${T^q}^\ll$ of $T^q$ obeys
\begin{align*}
{T^q}^\ll\big\vert_{\Lip_\comp(I^+_{K,\dim\mms}(\Sigma))} &\geq \vol\mres I^+_{K,\dim\mms}(\Sigma)\\
&\qquad\qquad - (\dim\mms-1)\int_{\Sigma}\textcolor{black}{\d\q(\alpha)\,\vol_\alpha\mres I^+_{K,\dim\mms}(\Sigma)}\\
&\qquad\qquad\qquad\qquad \textcolor{black}{\times\, l_\Sigma\,\sfC_{K,\dim\mms} \circ (\Len_l\,\mms_\alpha - l_\Sigma)}.
\end{align*}
\end{theorem}

\subsubsection{Volume and area estimates à la Heintze--Karcher} An influential inequality from Riemannian geometry is the sharp Heintze--Karcher inequality \cite{heintze-karcher1978} (see also Maeda \cite{maeda1978}). It estimates the volume of the radial future development of a hypersurface $\Sigma$ up to some threshold by the area and  mean curvature of $\Sigma$ as well as the ambient Ricci curvature and dimension. Sequella were shown by Bayle \cite{bayle2004} and Morgan \cite{morgan2005} on weighted Riemannian manifolds and Ketterer \cite{ketterer2020-heintze-karcher} for metric measure spaces with synthetic Ricci curvature bounds. 

On spacetimes, first alike volume-to-area comparison estimates for certain hypersurfaces  were shown by Treude--Grant \cite{treude-grant2013} for constant and Graf--Sormani \cite{graf-sormani2022} for variable mean curvature bounds. However, \cite{treude-grant2013} imposed  priori restrictions on the height of the radial future developments (both in terms of Ricci and mean curvature), while in \cite{graf-sormani2022} the positive part in the Jacobian function described below was only applied to the mean curvature, thus leading to a worse inequality.  As explained in \cref{Sub:Volume sing} below, these restrictions exclude the option to detect volume singularities without geodesic incompleteness. These issues were eliminated in the sharp Heintze--Karcher-type inequality by Braun \cite{braun2024+} surveyed here. It also holds in the nonsmooth context of TCD spaces.

Given any $K\in\R$, $\smash{\theta\in\R_+}$, and a real number $N$ larger than one, we define the Jacobian function
\begin{align*}
\sfJ_{K,N,H_0}(\theta) := \begin{cases}\displaystyle \Big[\!\cos\!\Big[\sqrt{\frac{K}{N-1}}\,\theta\Big] + \frac{H_0}{\sqrt{K(N-1)}}\,\cos\!\Big[\sqrt{\frac{K}{N-1}}\,\theta\Big]\Big]_+^{N-1} & \textnormal{if }K>0,\\
\displaystyle\Big[1+ \frac{H_0}{N-1}\,\theta\Big]_+^{N-1} & \textnormal{if }K=0,\\\displaystyle\Big[\!\cosh\!\Big[\sqrt{\frac{-K}{N-1}}\,\theta\Big] + \frac{H_0}{\sqrt{-K(N-1)}}\sinh\!\Big[\sqrt{\frac{-K}{N-1}}\,\theta\Big]\Big]_+^{N-1} & \textnormal{otherwise}.
\end{cases}
\end{align*}
The terms inside the positive parts summarize to $f' + H_0\,f/(N-1)$, where $f$ solves $f''+ K\,f/(N-1) = 0$ on $\R_+$ with initial conditions $f(0) = 0$ and $f'(0)=1$.

A simple consequence of \eqref{Eq:Semicon} is the following. 

\begin{lemma}[Density estimate \cite{ketterer2020-heintze-karcher}]\label{Le:Density} Let $\Sigma$ be an achronal, compact, spacelike hypersurface in $\mms$. Let $\q$ designate a disintegration of $\smash{\vol\mres I^+(\Sigma)}$ satisfying the conclusions from  \cref{Th:Disint}. Given $\alpha\in \Sigma$, define the right logarithmic derivative $\smash{H_0(\alpha) := (\log h_\alpha)'^+(0)}$. Then for every $\theta\in \mms_\alpha$,
\begin{align*}
h_\alpha(\theta) \leq \sfJ_{K,\dim\mms, H_0(\alpha)}(\theta)\,h_\alpha(0).
\end{align*}
\end{lemma}

How does this result relate to forward mean curvature? Inspired by a prior notion of Ketterer \cite{ketterer2020-heintze-karcher} in Riemannian signature, Braun \cite{braun2024+} \emph{defined} the synthetic forward mean curvature of a hypersurface $\Sigma$ as above at $\alpha\in \Sigma$ by $(\log h_\alpha)'^+(0)$. This number is nothing but the mean curvature of $\Sigma$ at $\alpha\in\Sigma$. (This relation between mean curvature and derivatives of the conditional densities has already been clarified by Cavalletti--Mondino \cite{cavalletti-mondino2020} based on  \cite{ketterer2020-heintze-karcher} as well.) Indeed, by the representation formulas in \cref{Th:Repr} only the behavior of $(\log h_\alpha)'^+$ near zero becomes relevant,  corresponding  to the restriction of the d'Alembertian $\Box l_\Sigma$ to $\Sigma$ --- which is the mean curvature of $\Sigma$ at the point in question. 

In the formulation of the next result, let $\smash{\hh_0 :=  h_\cdot(0)\,\q}$ denote the area measure of $\Sigma$, cf.~Cavalletti--Mondino \cite{cavalletti-mondino2020,cavalletti-mondino2022-review}. In addition, $\smash{\Sigma_{[0,t]} := (I^+(\Sigma)\cup\Sigma) \cap \{l_\Sigma \leq t\}}$ denotes the radial future development of $\Sigma$ up to temporal height $\smash{t\in\R_+}$. \MB{We recall \eqref{Eq:Ricvv} for the interpretation of lower boundedness of the Ricci curvature in all timelike directions.}

\begin{theorem}[Heintze--Karcher-type inequality \cite{treude-grant2013,graf-sormani2022,braun2024+}]\label{Th:HeintzeKarcher} Assume  $(\mms,g,\vol)$ constitutes a globally hyperbolic measured spacetime with $\Ric\geq K$ in all timelike directions, where $K\in\R$. Moreover, we suppose that $\Sigma$ is an achronal, compact, spacelike hypersurface in $\mms$ whose forward mean curvature is bounded from above by a real number $H_0$. Then for every $\smash{t\in\R_+}$,
\begin{align*}
\vol[\Sigma_{[0,t]}] \leq \hh_0[\Sigma]\int_0^t\sfJ_{K,N,H_0}(\theta)\d\theta.
\end{align*}
\end{theorem}

\begin{proof}[Proof sketch] By the above discussion,  the conditional density $h_\alpha(0)$ is positive for $\q$-a.e. $\alpha\in\Sigma$. Since $\smash{H_0(\alpha) \leq H_0}$ for $\q$-a.e.~$\alpha\in \Sigma$ by our  assumption, \cref{Le:Density} combines with basic monotonicity properties of the Jacobian function (using the identifications from \cref{Re:Identifications}) to yield
\begin{align*}
\vol[\Sigma_{[0,t]}] &= \int_\Sigma\int_{\mms_\alpha\cap [0,t]} h_\alpha \d\Haus_\alpha^1\d\q(\alpha)\\
&\leq \int_\Sigma h_\alpha(0) \int_{\mms_\alpha\cap [0,t]} \sfJ_{K,\dim\mms,H_0(\alpha)}\d\Haus_\alpha^1\d\q(\alpha)\\
&\leq \hh_0[\Sigma]\int_0^t \sfJ_{K,\dim\mms, H_0}(\theta)\d\theta.
\end{align*}
This terminates the proof. 
\end{proof}

\subsection{Volume singularity theorems}\label{Sub:Volume sing} The classical singularity theorems of Penrose \cite{penrose1965}, Hawking \cite{hawking1967}, and Hawking--Penrose \cite{hawking-penrose1970} are among the cornerstones of mathematical relativity. The singularity theorem of Hawking  implies timelike geodesic incompleteness under the presence of certain hypersurfaces (usually  with a mean curvature constraint that forces the focusing of radially emanating $l$-geodesics). The former means the existence of $l$-geodesics that extinct after a finite amount of proper time; think of a free-falling massive observer reaching the ``end'' of spacetime. By using \cref{Th:Disint}, Cavalletti--Mondino \cite{cavalletti-mondino2020} extended this result to abstract TCD spaces. The interested reader is referred to the reviews of Steinbauer \cite{steinbauer2023} and Cavalletti--Mondino \cite{cavalletti-mondino2022-review}, respectively. We also point out similar Hawking-type singularity theorems by Alexander--Graf--Kunzinger--Sämann \cite{alexander-graf-kunzinger-samann2023} under synthetic timelike sectional curvature bounds  and Graf--Kontou--Ohanyan--Schinnerl \cite{graf-kontou-ohanyan-schinnerl2022} and Braun--McCann \cite{braun-mccann2023} in smooth and nonsmooth settings with variable timelike Ricci curvature bounds, respectively.

To give nonexperts a grip on the shape of  a singularity theorem, we formulate Hawking's one. Given $K,H_0\in\R$ and a number $N$ larger than one, we set
\begin{align*}
\sfD_{K,N,H_0} := \begin{cases}\displaystyle \frac{\pi}{2}\sqrt{\frac{N-1}{K}} & \textnormal{if }K>0\textnormal{ and }H_0=0,\\
\displaystyle\sqrt{\frac{N-1}{K}}\cot^{-1}\!\Big[\frac{-H_0}{\sqrt{K(N-1)}}\Big] & \textnormal{if }K>0\textnormal{ and } H_0\neq 0,\\
\displaystyle \frac{N-1}{-H_0} &\textnormal{if }K=0 \textnormal{ and }H_0<0,\\
\displaystyle \sqrt{\frac{N-1}{-K}}\coth^{-1}\!\Big[\frac{-H_0}{\sqrt{-K(N-1)}}\Big] & \textnormal{if }K<0 \textnormal{ and }H_0 < -\sqrt{-K(N-1)}.
\end{cases}
\end{align*}
This is the first positive zero of the function $\sfJ_{K,N,H_0}$ from before \cref{Le:Density}. 

\begin{theorem}[Hawking singularity theorem \cite{hawking1967}]\label{Th:SyntheticHawking} Suppose $(\mms,g,\vol)$ is a globally hyperbolic measured spacetime satisfying $\Ric\geq K$ in all timelike directions, where $K\in\R$. Assume it contains an achronal, compact, spacelike hypersurface $\Sigma$ with forward mean curvature bounded from above by $H_0\in\R$. Lastly, assume
\begin{enumerate}[label=\textnormal{\alph*.}]
\item $K>0$ and no further restrictions on $H_0$,
\item $K=0$ and $H_0<0$, or
\item $K<0$ and $\smash{H_0< -\sqrt{-K(\dim\mms-1)}}$.
\end{enumerate}
Then $(\mms,g)$ is timelike geodesically incomplete; more precisely, $\smash{l_\Sigma \leq \sfD_{K,\dim\mms,H_0}}$.
\end{theorem}

This theorem and the Heintze--Karcher inequality \cite{heintze-karcher1978} inspired Treude--Grant \cite{treude-grant2013} to target ``volume versions'' of Hawking's singularity theorem. Recently, this motivated García-Heveling \cite{garcia-heveling2023-volume} to introduce the concept of a \emph{volume singularity}, complementing the traditional notion of a geodesic singularity.

\begin{definition}[Future volume incompleteness \cite{garcia-heveling2023-volume}] We  call the measured spacetime $(\mms,g,\vol)$ \emph{future volume incomplete} if for every $\varepsilon > 0$ there exists a point $x\in \mms$ such that $\smash{\vol[I^+(x)] \leq  \varepsilon}$.
\end{definition}

\MB{We call a point $x$ as above an \emph{$\varepsilon$-volume singularity} (and more broadly a \emph{volume singularity} if it is an $\varepsilon$-volume singularity for some $\varepsilon > 0$).}

The physical plausibility of this notion is argued in the introduction of  \cite{garcia-heveling2023-volume}. In a nutshell, by definition \MB{future volume incompleteness}   means the occurrence of chronological futures with volume smaller than Planck volume \MB{$V_\rmP$. At scales smaller than those defined by the  Planck level, quantum effects are expected to dominate. Certain theories of quantum gravity specifically suggest that spacetime becomes discrete or effectively two-dimensional at these scales. Thus, an observer reaching an $\varepsilon$-volume singularity, where $\varepsilon < V_\rmP$,} ``\textit{is no longer capable of making valid predictions using classical physics}'', as stated in \cite{garcia-heveling2023-volume}. This suggests a breakdown of general relativity and justifies the name ``singularity''. Connections  to the cosmic censorship conjecture of Penrose are also outlined in \cite{garcia-heveling2023-volume}*{§1}.

\begin{remark}[Volume vs.~geodesic incompleteness]\label{Re:Logical} Relevant regions of basic black hole models such as Schwarzschild, Reissner--Nordstr\o{}m, and Kerr are volume and geodesically  incomplete. On the other hand, the concepts of \MB{future} volume and \MB{future} geodesic incompleteness are \emph{logically independent}: in general, one property may hold while the other fails. We refer to García-Heveling \cite{garcia-heveling2023-volume} for details.\hfill$\blacksquare$
\end{remark}

Volume singularities become more tangible through the following basic property. 

\begin{lemma}[\textcolor{black}{Future} volume incompleteness and finite volumes  \cite{garcia-heveling2023-volume}] A measured spacetime $(\mms,g,\vol)$ is future volume incomplete if and only if it contains a point $x\in\mms$ such that  $\smash{\vol[I^+(x)]<\infty}$.
\end{lemma}

Recall the definition of the area measure $\hh_0$ from before  \cref{Th:HeintzeKarcher}.

\begin{theorem}[Hawking-type volume singularity theorem I  \cite{garcia-heveling2023-volume}]\label{Th:Hawkingvol} Let $(\mms,g,\vol)$ be a globally hyperbolic measured spacetime with $\Ric\geq K$ in every timelike direction, where $K\in\R$. Let it contain an achronal, compact, spacelike hypersurface $\Sigma$ whose forward  mean curvature is bounded from above by $H_0\in\R$. Lastly, assume
\begin{enumerate}[label=\textnormal{\alph*\textcolor{black}{.}}]
\item $K>0$ and no further restrictions on $H_0$,
\item $K=0$ and $H_0<0$, or
\item\label{La:EQUAL} $K<0$ and $\smash{H_0 \leq -\sqrt{-K(\dim\mms-1)}}$.
\end{enumerate}
Then $(\mms,g,\vol)$ is future volume incomplete; more  strongly, 
\begin{align*}
\vol[I^+(\Sigma)] \leq \hh_0[\Sigma]\begin{cases} 
\displaystyle\int_0^\infty\rme^{-\theta\sqrt{-K(\dim\mms-1)}}\d\theta & \textnormal{\textit{if} }K< 0 \textnormal{ \textit{and} }\\ & H_0 = -\sqrt{-K(\dim\mms-1)},\\
\displaystyle\int_0^{\sfD_{K,\dim\mms,H_0}} \sfJ_{K,\dim\mms,H_0}(\theta)\d\theta & \textnormal{\textit{otherwise}}.
\end{cases}
\end{align*}
\end{theorem}

Therefore, an entire hypersurface of points is a volume singularity.

The hypotheses of \cref{Th:Hawkingvol} are almost the same as for \cref{Th:SyntheticHawking} \emph{except} for the equality case from bullet point \ref{La:EQUAL} above. In all cases covered by the Hawking singularity \cref{Th:SyntheticHawking}, the proof of García-Heveling \cite{garcia-heveling2023-volume} shows  the finite volume of $\smash{I^+(\Sigma)}$ by ``integrating'' the uniformly bounded $l$-diameter. An alternative proof using the Heintze--Karcher-type inequality from \cref{Th:HeintzeKarcher} was given by Braun \cite{braun2024+}. It does not use \cref{Th:SyntheticHawking}, which underlines the logical independence of volume and geodesic incompleteness stated in \cref{Re:Logical}. Instead, it notes that the hypothesized relations between Ricci and mean curvature force the Jacobian function from before \cref{Le:Density} vanish identically after a finite threshold. Together with the finiteness of the area $\hh_0[\Sigma]$ --- as implied by compactness of $\Sigma$ ---, \cref{Th:HeintzeKarcher}  implies $\vol[I^+(\Sigma)]= \lim_{t\to \infty} \vol[\Sigma_{[0,t]}]<\infty$. This argument from  \cite{braun2024+} employing the abstract disintegration \cref{Th:Disint} also allow for a version of \cref{Th:Hawkingvol} on abstract TCD spaces in the same style.

\begin{remark}[Rigidity of \cref{Th:Hawkingvol}  \cite{andersson-galloway2002,galloway-woolgar2014}] We point out the  difference in the hypotheses of  \cref{Th:SyntheticHawking,Th:Hawkingvol} when $K$ is negative: while \cref{Th:SyntheticHawking} does not yield a geodesic singularity  when $\smash{H_0 = -\sqrt{-K(\dim\mms-1)}}$, \cref{Th:Hawkingvol} does predict a volume singularity. An  example of a measured spacetime with a compact Cauchy hypersurface satisfying these curvature constraints which is future volume incomplete yet not all $l$-geodesics are future complete can be constructed using warped products, cf.~\cref{Ex:Model}. As shown by Andersson--Galloway \cite{andersson-galloway2002} and later Galloway--Woolgar \cite{galloway-woolgar2014}, this setting is rigid. \hfill$\blacksquare$
\end{remark}

Based upon the same idea, Braun \cite{braun2024+} established previously unknown further volume singularity theorems under variable timelike Ricci curvature bounds that have been anticipated by García-Heveling \cite{garcia-heveling2023-volume}. To simplify the presentation, we only report the following result. It is a volume version of a singularity theorem of Frankel--Galloway \cite{frankel-galloway1981} later synthesized using \cref{Th:Disint} by Braun--McCann \cite{braun-mccann2023}. 

\begin{theorem}[Hawking-type volume singularity theorem II \cite{braun2024+}] Let $(\mms,g,\vol)$ be a measured spacetime such that $\Ric\geq k$ in all timelike directions, where $k$ is a lower semicontinuous function on $\mms$. Assume it contains an achronal, compact, spacelike hypersurface $\Sigma$  whose forward mean curvature is bounded from above by $H_0\in\R$. Assume $\varepsilon >0$ and $c>0$ obey $2c+ \varepsilon\, H_0/(\dim\mms-1) < 0$ and on $I^+(\Sigma)$, we have the pointwise inequality
\begin{align*}
k_- \leq c\,\big[\varepsilon^{-2}\wedge l_\Sigma^{-2}\big].
\end{align*}
Then $(\mms,g,\vol)$ is future volume incomplete; more strongly,  $\smash{\vol[I^+(\Sigma)] < \infty}$.
\end{theorem}

We note again the locally uniform lower bound on the timelike Ricci curvature holds e.g.~when the null energy condition of Penrose is in place, cf.~McCann \cite{mccann2020}.

\section{Open problems and future research}

In this last section, we collect several future challenges and conjectures arising from the material surveyed in this article. Some of these sharpen problems that have already been posed (especially in Cavalletti--Mondino's review \cite{cavalletti-mondino2022-review}). Others are regarded as natural applications by experts in optimal transport and metric geometry that we wish to popularize to a broader audience.

\subsection{D'Alembertian and comparison theory in different settings} The distributional approach to the d'Alembertian and  the adjacent  comparison theory we have surveyed in this paper   are based on the synthetic approach to timelike lower Ricci curvature bounds by Cavalletti--Mondino \cite{cavalletti-mondino2020}. There are related settings with abstract notions of curvature-dimension bounds where analogous progress could prove useful for the development of the respective theory. 

\subsubsection{TCBB spaces}\label{Sub:TCBB} The first are metric measure spacetimes with synthetic timelike sectional cur\-vature lower bounds, briefly TCBB spaces (an acronym for ``timelike curvature bounded from below''). Inspired by the metric theory of Alexandrov spaces and first attempts to their Lorentzification by Andersson-Howard \cite{andersson-howard1998} and Alexander--Bishop \cite{alexander-bishop2008}, they were introduced by Kunzinger--Sämann \cite{kunzinger-samann2018} and studied further by Minguzzi--Suhr \cite{minguzzi-suhr2022} and Bykov--Minguzzi--Suhr \cite{bykov-minguzzi-suhr2024+}. Experts in metric geometry conjecture TCBB spaces are TCD spaces (cf.~the survey of Cavalletti--Mondino \cite{cavalletti-mondino2022-review}); thus, in principle  the theory surveyed in this paper should eventually apply to TCBB spaces. However, although the arc from sectional to Ricci curvature bounds is straightforward in classical differential geometry, the proof of this implication in Riemannian signature by Petrunin \cite{petrunin2011} and  Zhang--Zhu \cite{zhang-zhu2010} is quite sophisticated. Hence, it might prove useful to target the surveyed analysis separately, e.g.~by understanding optimal transport on TCBB spaces first following Bertrand \cite{bertrand2008}. A first step in the direction of quantitative estimates on TCBB spaces is the sharp and rigid Bonnet--Myers theorem of Beran--Napper--Rott \cite{beran-napper-rott2023+} and  Beran \cite{beran2024+}.


A probably related question is whether the d'Alembertian (or its nonlinear $p$-version) admits a coordinate representation on TCBB spaces. Analogous formulas on Alexandrov spaces are due to Kuwae--Shioya \cite{kuwae-shioya2007} (which they used to generalize  Laplace comparison theorems of Petrunin \cite{petrunin2003} and von Renesse \cite{vonrenesse2004}). This representation is based on the fine structure of Alexandrov spaces extrapolated by Burago--Gromov--Perelman \cite{burago-gromov-perelman1992} and Otsu--Shioya \cite{otsu-shioya1994}, which has no counterpart on TCBB spaces thus far, despite first developments by Beran--Sämann \cite{beran-samann2023}.

\begin{conjecture}[Coordinate representation]\label{Conj:Coord} Any globally hyperbolic finite-di\-men\-sional TCBB space $(\mms,l)$  admits a conegligible subset $\mms^*$ which has the structure of a smooth topological manifold and has a continuous Lorentzian metric $g$ of locally bounded variation. On $\mms^* \cap (I^+(o)\setminus \TCut^+(o))$, the distributional $p$-d'Alembertian of a Lorentz distance function $\smash{l_o}$ of a point $o\in\mms$ has the subsequent pointwise representation for every nonzero number $p$ less than one:
\begin{align*}
\frac{1}{\sqrt{\vert\!\det g\vert}}\,\frac{\partial}{\partial x^i}\Big[\sqrt{\vert\!\det g\vert}\,g^{ij}\,\frac{\partial l_o}{\partial x^j}\Big].
\end{align*}
\end{conjecture}

Here, ``finite-dimensional'' and ``conegligible'' are  understood with respect to the Hausdorff measure on metric measure spacetimes of McCann--Sämann \cite{mccann-samann2022}. Already since the introduction of  TCBB spaces, the first half of \cref{Conj:Coord} is a  general belief  of   experts in synthetic Lorentzian geometry.

Another conceivable  way of attacking comparison theory --- more in line with the localization approach to the distributional $p$-d'Alembertian of Braun \cite{braun2024+} --- might be to transfer Cavalletti--Mondino's disintegration \cref{Th:Disint}  \cite{cavalletti-mondino2020} to TCBB spaces. The setting seems tailor made for this purpose. Indeed, this technique requires the base space to be timelike nonbranching after Cavalletti--Mondino \cite{cavalletti-mondino2020} (but also works under a weaker version thereof proposed by Braun \cite{braun2024+}). As shown by Kunzinger--Sämann \cite{kunzinger-samann2018} and later Minguzzi--Suhr \cite{minguzzi-suhr2022}, TCBB spaces are indeed timelike nonbranching.

\subsubsection{TCBA spaces} The next setting are  synthetic upper curvature bounds.  These could e.g.~be the synthetic timelike Ricci curvature upper bounds of Mondino--Suhr \cite{mondino-suhr2022} inspired from Sturm \cite{sturm2021-upper} or TCBA spaces (an acronym for ``timelike curvature bounded from above'') as  introduced by Kunzinger--Sämann \cite{kunzinger-samann2018}. Some d'Alembert, area, and volume comparison results on spacetimes with timelike \emph{Ricci} curvature bounded from above were shown by Treude--Grant \cite{treude-grant2013}*{§A}. It might be interesting to find a synthetic access to these. One challenge here is that the synthetic timelike Ricci curvature upper bounds from \cite{mondino-suhr2022} only give an approximate control on the timelike Ricci tensor along short $l$-geodesics.

\subsubsection{Negative-dimensional TCD spaces} A last possible setting are TCD spaces with ``negative'' dimension (in a suitable synthetic sense). They enjoy increasing interest in spacetime geometry (triggered by Brans--Dicke, Kaluza--Klein, or certain aspects of string theory), cf.~e.g. Woolgar--Wylie \cite{woolgar-wylie2016,woolgar-wylie2018} or De Luca--De Ponti--Mondino--Tomasiello \cite{deluca-deponti-mondino-tomasiello2023} for more recent references. The TCD condition with negative dimensional parameter has been introduced by Braun--Ohta \cite{braun-ohta2024} inspired by Ohta \cite{ohta2016-negative} in Riemannian signature. For Riemannian comparison results in this direction, we refer to  Wylie--Yeroshkin \cite{wylie-yeroshkin2016+}, Kuwae--Li \cite{kuwae-li2022}, and Lu--Minguzzi--Ohta \cite{lu-minguzzi-ohta2022-range}. The latter reference also establishes analogous d'Alembert comparison theorems for Finsler spacetimes with ``negative'' dimension. This indicates a synthetic access to these estimates including the future timelike cut locus through the methods we surveyed. However,  the usual results on the solvability of the Monge problem in ``negative-dimensional'' optimal transport  --- needed to run e.g.~the localization paradigm for \cref{Th:Disint} --- become  more delicate by the shape of the involved entropy functionals, cf.~e.g.~Magnabosco--Rigoni \cite{magnabosco-rigoni2023} in Riemannian signature.

\subsection{Analytic and probabilistic aspects of the $p$-d'Alembertian} In Riemannian geometry, the fascinating  interplay of the Laplacian and Ricci curvature is certified by two concepts  induced by the Laplace operator. One is analytic, one is  probabilistic: \emph{heat flow} and \emph{Brownian motion}. As pointed out by Yau, Elworthy, Malliavin, Bismut, Davies, and others, lower bounds on the Ricci curvature lead to many functional inequalities, coupling estimates,  exit time bounds, derivative formulas, etc. A  number of these results has been generalized to metric measure spaces with synthetic Ricci curvature bounds, cf.~the ECM survey of Sturm \cite{sturm2023}.

It is interesting to introduce similar objects and machineries in the Lorentzian setting and study their consequences.

Making the concept of ``heat flow'' rigorous in Lorentzian signature has already been proposed more broadly in the review of Cavalletti--Mondino \cite{cavalletti-mondino2022-review}. They suggest the construction of a Lorentzian theory of gradient flows, inspired by its successful predecessor in  metric geometry (cf.~e.g.~Ambrosio--Gigli--Savaré \cite{ambrosio-gigli-savare2008}). It is likely the $p$-d'Alembertian \eqref{Eq:Operatorrr} is the correct object to consider in this direction. This educated guess is based on its natural link  to the Lagrangians appearing in our setting and supporting prior relations in metric measure geometry pointed out e.g.~by Jordan--Kinderlehrer--Otto \cite{jordan-kinderlehrer-otto1998}, Otto \cite{otto2001}, Ambrosio--Gigli--Savaré \cite{ambrosio-gigli-savare2014-calculus}, Erbar--Kuwada--Sturm \cite{erbar-kuwada-sturm2015}, and Ambrosio--Mondino--Savaré \cite{ambrosio-mondino-savare2019}. Such a gradient flow theory is work in progress by Braun--Gigli--McCann--Vincini \cite{braun-gigli-mccann-vincini+}. It will still remain interesting and relevant to study its consequences especially on structures with timelike Ricci curvature bounds. 

The probabilistic part appears more challenging. We point out the existence of ``Brownian motion'' on spacetimes (often called ``relativistic diffusion'') shown by  Dudley \cite{dudley1966} and Franchi--Le Jan \cite{franchi-lejan2007}. However, the construction of this process differs drastically from its Riemannian analog: it is built on phase space (instead of the base manifold alone) and is obtained by solving an ODE, not an SDE. This pathwise construction  effectively makes  relativistic diffusions   nonprobabilistic. On the other hand, surprisingly recently Barbu--Rehmeier--Röckner  \cite{barbu-rehmeier-rockner2024+} proposed a Markov process on Euclidean space they called ``$p$-Brownian motion''. It relates to the customary $p$-Laplacian in the same way Brownian motion relates to the usual Laplacian. Their theory, based on McKean--Vlasov SDEs, would be interesting to generalize to Lorentzian signature in light of the operator \eqref{Eq:Operatorrr}.

\subsection{Splitting theorem under infinitesimal  Minkowskianity} As stated in the introduction, the  Lorentzian splitting theorem of Eschenburg  \cite{eschenburg1988}, Galloway \cite{galloway1989-splitting}, and Newman \cite{newman1990} has recently been reproven with  considerably simpler arguments by Braun--Gigli--McCann--Ohanyan--Sämann \cite{braun-gigli-mccann-ohanyan-samann+}. The argument is based on the elliptic nature of the $p$-d'Alembertian \eqref{Eq:Operatorrr} first observed by Beran--Braun--Calisti--Gigli--McCann--Ohanyan--Rott--Sämann \cite{beran-braun-calisti-gigli-mccann-ohanyan-rott-samann+-}. These methods have recently been used by Caponio--Ohanyan--Ohta  \cite{caponio-ohanyan-ohta2024+}  to generalize a splitting theorem for Finsler spacetimes by Lu--Minguzzi--Ohta \cite{lu-minguzzi-ohta2023-splitting}. With a more classical argument (cf.~Burago--Burago--Ivanov \cite{burago-burago-ivanov2001}), the splitting theorem for nonnegatively curved TCBB spaces --- a conjecturally stronger hypothesis than the TCD condition after \cref{Sub:TCBB} --- has been established by Beran--Ohanyan--Rott--Solis \cite{beran-ohanyan-rott-solis2023}.

As shown by Braun--Ohta \cite{braun-ohta2024}, the TCD  condition of Cavalletti--Mondino \cite{cavalletti-mondino2020} covers Finsler spacetimes; the analogous implication in Riemannian signature is due to Ohta \cite{ohta2009-finsler}. Thus, it is clear that TCD spaces do not split in general. This necessitates a further hypothesis which, when coupled with an appropriate TCD condition, conjecturally entails a nonsmooth Lorentzian splitting theorem. To this aim, Beran--Braun--Calisti--Gigli--McCann--Ohanyan--Rott--Sämann \cite{beran-braun-calisti-gigli-mccann-ohanyan-rott-samann+-} introduced the following ``quadraticity'' property of a general metric measure spacetime. It was inspired by Gigli's infinitesimally Hilbertian metric measure spaces \cite{gigli2015}, which eventually lead to his celebrated nonsmooth splitting theorem \cite{gigli2013}.

\begin{definition}[Infinitesimal Minkowskianity \cite{beran-braun-calisti-gigli-mccann-ohanyan-rott-samann+-}]\label{Def:InfMink} A metric measure spacetime $(\mms,l,\meas)$ is called \emph{infinitesimally Minkowskian} if for all $l$-causal functions $\u$ and $\v$ on $\mms$, the following parallelogram identity holds: 
\begin{align*}
2\big\vert\rmd\u \big\vert_*^2 + 2\big\vert\rmd(\u+\v) \big\vert_*^2 =\big\vert\rmd\v\big\vert_*^2 + \big\vert\rmd(2\u+\v)\big\vert_*^2\quad\meas\textnormal{-a.e.}
\end{align*}
\end{definition}

Here, $\smash{\vert\rmd\cdot\vert_*}$ is the so-called ``maximal weak subslope'' of the $l$-causal  function in question, defined in \cite{beran-braun-calisti-gigli-mccann-ohanyan-rott-samann+-} by entirely abstract means.

The work \cite{beran-braun-calisti-gigli-mccann-ohanyan-rott-samann+-} already shows a Finsler spacetime is infinitesimally Minkowskian if and only if its Finsler structure on the  tangent bundle comes from a Lorentzian metric. Furthermore, given  $K\in\R$ and a number $N$ no less than one, \cite{beran-braun-calisti-gigli-mccann-ohanyan-rott-samann+-} define an $\LTCD_q^e(K,N)$ space to be an infinitesimally Minkowskian $\smash{\TCD_q^e(K,N)}$ space, comparable to the RCD spaces of Gigli \cite{gigli2015}.

\begin{conjecture}[Nonsmooth splitting theorem]\label{Conj:Splitting} Let $(\mms,l,\meas)$  be a globally hyperbolic $\smash{\LTCD_q^e(0,N)}$ metric measure spacetime, where $N$ is a real number no less than one. Assume that it contains a line. Then $(\mms,l,\meas)$ is isomorphic as a metric measure spacetime to the generalized cone $(\R\times \sfN, \vert\cdot -\cdot \vert\times \mathsf{d}, \Leb^1 \otimes \mathfrak{n})$, where $(\sfN, \sfd,\mathfrak{n})$ is an $\RCD(0,N-1)$ metric measure space.
\end{conjecture}

A line means an order-isometric embedding of $(\R, \leq)$ into $(\mms,\leq)$. Generalized cones are understood in the  sense of Alexander--Graf--Kunzinger--Sämann  \cite{alexander-graf-kunzinger-samann2023}.

This conjecture is a general belief by experts in metric geometry. It is well-motivated from its Riemannian predecessor of Gigli \cite{gigli2015}. However, an entire set of techniques still needs to be developed before the successful resolution of this program. Crucial ingredients missing to date are a nonsmooth maximum principle for the distributional $p$-d'Alembertian (e.g.~directed from the smooth case by Braun--Gigli--McCann--Ohanyan--Sämann \cite{braun-gigli-mccann-ohanyan-samann+}) and a nonsmooth analog of Mondino--Suhr's $p$-Bochner inequality \cite{mondino-suhr2022} (which  has already been set as a task in the review of Cavalletti--Mondino \cite{cavalletti-mondino2022-review}). A further challenge are many subtle regularity issues of the Busemann function already arising in classical spacetime geometry (as pointed out e.g.~by Galloway--Horta \cite{galloway-horta1996}) that require proper adaptation to  metric measure spacetimes. As pointed out to me by Nicola Gigli, the Bochner-type inequality of Braun \cite{braun2024+} stated in the introduction only controls the ``Hessian'' of the function in question in tangential directions; however, the splitting theorem of  \cref{Conj:Splitting} will require its control in all directions. Moreover, the results of Galloway--Horta \cite{galloway-horta1996} indicate it will be necessary to understand the (still subtle) relation between time separation function and topology.

First soft steps towards this endeavor could be to understand the infinitesimal Minkowskianity of TCBB spaces (which should follow by adapting Braun--Ohta \cite{braun-ohta2024} from the smooth case) and of generalized cones, respectively.

The resolution of \cref{Conj:Splitting} would open the door to a rich structure theory for LTCD spaces, guided by the cornerstones in this direction for RCD spaces by Mondino--Naber \cite{mondino-naber2019} and Brué--Semola \cite{brue-semola2020-constancy}. Lorentzian versions of the almost splitting theorem by Cheeger--Colding \cite{cheeger-colding1996-almost} and their structural study of Ricci limit spaces \cite{cheeger-colding1997-i,cheeger-colding2000-ii,cheeger-colding2000-iii} are particularly interesting.

\bibliographystyle{amsrefs}


\end{document}